\documentclass[11pt,titlepage]{article}
\usepackage{latexsym,graphicx,amsthm,amssymb,bm,url,amsmath}
\usepackage[multiple]{footmisc}
\usepackage{pifont}
\usepackage{secdot}
\usepackage{xcolor}
\usepackage{bbm}
\usepackage{multirow}
\usepackage{float}
\usepackage{cancel}
\usepackage[shortlabels]{enumitem}
\usepackage[ruled,noend]{algorithm2e} 
\usepackage{setspace}
\SetAlgorithmName{Procedure}{algorithmautorefname}

\usepackage[colorinlistoftodos,textsize=tiny
]{todonotes}
            
\definecolor{agocolor}{rgb}{1.0, 0.0, 0.0}

\usepackage{threeparttable}
\usepackage[inline]{trackchanges}
\usepackage{caption}
\usepackage{subcaption}
\usepackage{hyperref}

\usepackage[margin=1in]{geometry}

\usepackage{natbib}
\bibpunct[, ]{(}{)}{,}{a}{}{,}%
\usepackage{thmtools}
\usepackage{thm-restate}
\declaretheorem{lemma}

\declaretheorem{assumption}
\declaretheorem{definition}

\def\argmax{\mathop{\rm argmax}}
\def\argmin{\mathop{\rm argmin}}

\newcommand{\cmark}{\ding{51}}%
\newcommand{\xmark}{\ding{55}}%
\graphicspath{ {../} }

\makeatletter
\newcommand{\RemoveAlgoNumber}{\renewcommand{\fnum@algocf}{\AlCapSty{\AlCapFnt\algorithmcfname}}}
\makeatother

\usepackage{soul}

\def\mcal#1{\mathcal{#1}}

\def\bigT{\mcal{T}}
\def\bigV{\mcal{V}}
\def\bigS{\mcal{S}}

\def\bigU{\mcal{U}}
\newcommand{\eq}{\[}
\newcommand{\en}{\]}

\begin{document}

\begin{center}
{\LARGE Virtual Trading in Multi-Settlement Electricity Markets}\\
\end{center}




\begin{center}
{\large
 \mbox{Agostino Capponi}\\
 \mbox{{\small Department of Industrial Engineering and Operations Research, Columbia University, New York, NY USA}}
}
\mbox{{\small ac3827@columbia.edu}}
\end{center}

\begin{center}
{\large
\mbox{Garud Iyengar}\\
\mbox{{\small Department of Industrial Engineering and Operations Research, Columbia University, New York, NY USA}}\\
\mbox{{\small garud@ieor.columbia.edu}}
}
\end{center}

\begin{center}
{\large
\mbox{Bo Yang}\\ 
\mbox{{\small Department of Industrial Engineering and Decision Analytics, The Hong Kong University of Science and Technology}}\\
\mbox{{\small Hong Kong SAR, China}}}\\
\mbox{{\small yangb@ust.hk}}
\end{center}

\begin{center}
{\large
\mbox{Daniel Bienstock}\\
\mbox{{\small Department of Industrial Engineering and Operations Research, Columbia University, New York, NY USA}}\\
\mbox{{\small Department of Applied Mathematics, Columbia University, New York, NY USA} }}\\ 
\mbox{{\small dano@ieor.columbia.edu}}
\end{center}

\vspace{0.2cm}
\begin{center}
{{\bf Abstract}}
\end{center}

{{\noindent
In the Day-Ahead (DA) market, suppliers sell and load-serving entities
(LSEs) purchase energy commitments, with both sides adjusting for
imbalances between contracted and actual deliveries 
in the Real-Time (RT) market. We develop a supply function equilibrium
model to study how virtual trading—speculating on DA–RT price spreads
without physical delivery—affects market efficiency. Without virtual
trading, LSEs underbid relative to actual demand in the DA market, pushing
DA prices below expected RT prices. Virtual trading narrows, and 
in the limit of large number
traders can eliminates, this price gap. {However, it does not induce quantity alignment: DA-cleared demand remains below true expected demand, as price alignment makes the LSE indifferent between markets and prompts it to reduce DA bids to avoid over-purchasing.} 
Renewable energy suppliers cannot offset these strategic
distortions. We provide empirical support to our main model implications
using data from the California and New York Independent System Operators.


\looseness=-1
}
}

\medskip
{{\noindent {\bf Keywords:} Virtual trading, electricity market, renewable energy, supply function equilibrium, market design}}

\baselineskip 20pt plus .3pt minus .1pt


\interfootnotelinepenalty=10000

\section{Introduction}
\label{sc:Intro}
The U.S. electricity markets %
utilize a dual settlement system consisting of  
a day-ahead (DA) and real-time (RT) markets 
for the procurement and transaction of  electricity.
Market participants, including suppliers and load-serving entities (LSEs), make electricity reservations in the DA market based on anticipated conditions in the RT market. Their planning takes into account expected demand and the projected availability of renewable energy sources. Any differences between the reserved amounts and the actual electricity deliveries are adjusted in the RT market. This structure is designed to accommodate the diverse startup and response times of generation resources, and offer participants the possibility to hedge against fluctuations in prices.

However, despite its advantages, this two settlement structure can introduce inefficiencies into the market. One common problem is the strategic price manipulation by LSEs, particularly those holding significant market shares at specific grid nodes. These large LSEs might alter their bid quantities in the DA market to lower DA prices (see, e.g., \citealt[pages 115 and 121]{isone2004coldsnap}, and \citealt{jha2023can}). This tactic, while advantageous for the manipulating LSEs, is considered ``extremely detrimental'' to the overall market's long-term vitality as it hinders other participants' capacity to effectively hedge risks (\citealt[page 26]{pjm2015virtual}).

{\it Virtual trading}  was 
implemented in electricity markets to mitigate market inefficiencies, specifically the disparity between DA and RT prices~(see, e.g., \citealt{pjm2015virtual} and \citealt{hogan2016virtual}).  Virtual bids, purely financial in nature, allow participants to assume financial positions in the DA market, which are subsequently cash settled or reversed in the RT market. The term ``virtual" indicates that no physical delivery or consumption of electricity is required. There are two major types of virtual bids: increment (INC) offers and decrement (DEC) bids.\footnote{Electricity markets also allow Up-to-Congestion transactions. This transaction is economically equivalent to INC and DEC bids submitted simultaneously at two different power grid nodes.} INC bids involve a virtual sale of electricity in the DA market, followed by a virtual purchase of the same amount in the RT market. Hence, the INC bidder competes with the suppliers in the DA market.
In contrast, a DEC bidder competes with LSEs by virtually purchasing electricity in the DA market and then selling this amount in the RT market.
The importance of virtual bidding in mitigating market power, especially that of LSEs, has been highlighted by PJM, an independent system operator (ISO) that coordinates the movement of wholesale electricity for 13 states and the District of Columbia. In their report (see \citealt[page 31]{pjm2015virtual}), they write:
\begin{quote}
``\textit{The opportunities for INCs to mitigate supply-side market power are few, whereas the market design features of the Day-Ahead Market that result in persistent underbidding of load in the Day-Ahead Market provide opportunities for DEC bids to mitigate demand side market power much more often. This is likely a primary driver in the reason cleared DEC volume exceeds cleared INC volume over the previous three planning years by a ratio of 1.5 to 1.}"\looseness=-1
\end{quote}

{
Empirical evidence indicates that virtual trading has been effective in narrowing the price gap between DA and RT markets \citep{jha2023can,li2015efficiency}. Despite this, a detailed theoretical examination of the mechanisms leading to this price convergence is lacking. We introduce the first comprehensive framework aimed at explaining how virtual bidding contributes to narrowing the price differences between the DA and RT markets, as well as to aligning the trading quantities within these markets.  Our model accounts for the growing penetration of wind and solar power—now supplying 30\% of California’s electricity \citep{CEC2022total}—by explicitly incorporating renewable energy suppliers and their influence on market dynamics. With the flexibility to adjust output in response to weather conditions, renewable providers hold strategic advantages over traditional fossil fuel suppliers, influencing their market participation and behavior.

\looseness=-1


We model the two-settlement market as a two stage game involving renewable and conventional suppliers, load serving entities, and virtual traders.\footnote{Our model is stylized, and in practice, markets are subject to numerous ad hoc adjustments and interventions. For example, bilateral agreements between generators, loads, and reserves are beyond the scope of this study and are therefore omitted.} In the first stage, which we refer to as the {\it DA market stage}, both demand and the capacity of renewable suppliers are uncertain. However, the probability distributions governing these uncertainties are assumed to be common knowledge among all participants.\footnote{In practice, electricity markets have 24 hourly DA markets per day. Participants submit bids for each hour, and repeat this decision making process 24 times during the day. We consider one of these periods, consistently with standard modeling approaches in the literature (see, e.g., \citealt{sunar2019strategic} and references therein).} 
The participants submit supply and demand bids, anticipating that they will be able to adjust their bids in the second stage, i.e., {\it RT market stage}, when the realized demand and capacity becomes common knowledge.\footnote{At the commencement of the RT market, the estimates for demand and capacities within the RT market period are highly accurate. In our model, we make the assumption that these quantities are perfectly known.}

In our model, both renewable and conventional suppliers submit supply
function bids (referred to as cost curves in industry parlance)
both in the DA and RT markets. These bids specify the
quantities they seek to produce as a function of the market price.  

Renewable suppliers—whose capacities are subject to weather-driven uncertainty—have the opportunity to revise their supply functions in the RT market. In contrast, the DA supply functions submitted by conventional suppliers are fully observable to all market participants and are treated as binding. As a result, conventional suppliers are not permitted to revise their bids in the RT market. 
The assumption reflects the fact that the  
ISO can readily ascertain the marginal costs and
capacities of the conventional suppliers using historical data 
(see, e.g., \citealt[page
30]{NYISO_RLM}). 
In addition,
we assume that the conventional supplier is uncapacitated, consistent with data from the California Independent System Operator
(CAISO) on power outages (see, e.g., \citealt[page 2]{CAISO_Outage}).\looseness = -1

Virtual bidders participate exclusively in the DA market. If they expect the RT market price to fall below the DA market price, they take a short position by submitting a supply function. Conversely, if they anticipate the RT market price will exceed the DA price, they take a long position by submitting a negative supply function. In both cases, their bids are settled at RT market prices.


We characterize the {\it supply function equilibria} (SFEs) in the DA and RT markets. Our equilibrium analysis reveals that, in the absence of virtual trading and renewable suppliers, LSE strategically bids to push the price in the DA market below the expected RT price. By strategically underbidding—contrary to the truthful bidding expected by the ISO—LSEs can procure part of their demand at prices below those implied by truthful bids, effectively shifting surplus from conventional suppliers to themselves. We show that the presence of virtual trading in the market reduces the price gap between the two markets, and leads to a {\it full alignment} of DA and RT prices in the limit when there is a large number of virtual traders. 
{By contrast, without virtual bidders, the presence of strategic renewable suppliers alone is insufficient to eliminate the LSE’s market power: the DA price remains below the expected RT price. Perfect competition among virtual traders closes the gap even when strategic renewable suppliers are active.}
 
While virtual trading facilitates convergence between the DA and RT prices, it does \emph{not} imply that 
there is quantity alignment, i.e., the demand cleared 
in the DA market need \emph{not} match
the true expected demand. With virtual bidders present, the cleared DA demand is {\it lower} than in the absence of virtual bidding. Once DA and RT prices align, the LSE becomes indifferent between sourcing electricity in either market and reduces its DA bid to avoid over-reserving relative to actual demand.
In this case, virtual traders offset the LSE’s shortfall in the DA market through DEC bids.\looseness = -1
 

We provide empirical support for some of our model implications using historical market data from the California ISO (CAISO) and the
New York ISO (NYISO),
leveraging the availability of data from periods both before and after the
implementation of virtual trading. 
Consistent with the
theoretical implications of the model, our findings reveal that virtual
trading narrows the price gap between the DA and RT markets in both CAISO
and NYISO.  
The average price gap experiences significant reductions throughout the
day, with CAISO and NYISO observing average decreases of 313\% and 240\%
respectively.  
Additionally, virtual trading leads LSEs to reduce their bid quantities in the DA market, lowering the average cleared demand across all hours of the day by $25\%$ in the DA market and $33\%$ in the RT market.


\subsection{Literature Review}
\label{sc:Literature}
Our work contributes to the body of literature on virtual trading in the
electricity market. Previous studies, such as \cite{mather2017virtual} and
\cite{tang2016model}, have demonstrated that
as the number of virtual traders increases,  the quantity of electricity
cleared in the DA market approaches the socially optimal DA schedule. 
\cite{mather2017virtual} focus on conventional suppliers, virtual traders,
and the ISO, with the latter two acting
strategically. \cite{tang2016model} consider strategic virtual traders,
assuming that suppliers and LSEs submit truthful bids for their supply and
demand offers and attribute price discrepancies to estimation
errors. Unlike these studies, in our model renewable suppliers, LSEs, and virtual traders are all strategic agents,
and the LSE's market power is identified as the primary factor behind the
deviation of DA from RT prices. Moreover, we show that while virtual trading aligns prices, it can also cause quantity misalignment by reducing the LSE’s exposure to overbidding risk. In contrast to \citet{mather2017virtual} and \citet{tang2016model}, who model the electricity market as a quantity game, we develop a supply function equilibrium framework that endogenizes both prices and quantities. Rather than imposing quadratic supply functions for conventional suppliers, we derive them endogenously.

{A related study by 
\citet{you2019role} finds that virtual trading can eliminate DA–RT price discrepancies with sufficient traders. 
Our work extends theirs by embedding the analysis in a broader framework that incorporates stochastic demand and strategic renewable suppliers with random capacities, alongside a strategic LSE. In contrast, \citet{you2019role} assume deterministic demand and model only the LSE as strategic. We further adopt a supply function equilibrium framework that endogenously derives supply functions, rather than a quantity competition model with quadratic costs. This richer setting illustrates the mechanism through which virtual trading can induce strategic quantity misalignment between DA bids and true demand. Table~\ref{tab:vt_lit_comprs} summarizes how our work relates to the existing literature on virtual trading.} 
\begingroup
\setlength{\tabcolsep}{2pt} 
\renewcommand{\arraystretch}{1} 
\begin{table}[h!]
\begin{tabular}{ccccccc}
\hline
Paper\textbackslash Feature & Supply Function & General & Random & Renewable & Strategic & Capacity \\
& Equilibrium & Function Form & Demand & Supplier & LSE & Constraint \\ \hline
\cite{mather2017virtual}  & \xmark & \xmark & \cmark & \xmark & \xmark & \xmark \\
\cite{you2019role}        & \xmark & \xmark & \xmark & \xmark & \cmark & \xmark \\
\cite{tang2016model}      & \xmark & \xmark & \xmark & \xmark & \xmark & \xmark \\
Our paper     & \cmark & \cmark & \cmark & \cmark & \cmark & \cmark \\ \hline
\end{tabular}
\caption{Position of our paper with respect to literature on virtual trading}
\label{tab:vt_lit_comprs}
\end{table}
\endgroup

A few studies have analyzed virtual trading from an  empirical
perspective. \cite{jha2023can} examine the DA and RT prices from
California ISO before and after the implementation of virtual trading in
2011. They empirically demonstrate that virtual trading narrows the price
discrepancy between these two markets. They also find that virtual trading
effectively reduces trading costs in the electricity market and improves
operating efficiency.  
\cite{li2015efficiency} assess the performance of virtual traders' trading
strategies in the California electric power market. Their work shows that
virtual trading has become less profitable since its initiation, proving
its effectiveness in the electricity market. \cite{birge2018limits}
investigate the limit of virtual trading, noting that financial
institutions may use it to manipulate market prices via betting in the
opposite direction of the pricing gap. As a result, they may incur losses
in virtual trading, which are hedged through positions in  other
derivatives, such as financial transmission
rights. \cite{HOPKINS2020104818} finds that financial institutions can
also use congestion revenue rights, another kind of financial derivative,
to achieve similar effects. Our work provides a theoretical interpretation
of the results in \cite{jha2023can} and \cite{li2015efficiency}.\looseness
=-1\footnote{A separate stream f literature has studied energy commodity
  trading \citep{wu2005competitive,
    secomandi2014optimal,trivella2023meeting,
    dong2007equilibrium,mendelson2007strategic,pei2011sourcing}. These
  studies do not examine the effects of virtual trading and the
  integration of renewable energy suppliers on price dynamics as well as
  DA demand of LSEs in the electricity market.} 

Our paper also contributes to the stream of research on renewable energy markets. \cite{peura2021renewable} 
investigate the impact of wind power generation on electricity prices. 
\cite{sunar2019strategic}  
examine the strategic commitments of renewable energy suppliers to their production schedules. 
\cite{al2017understanding} explore the effects of different levels of production flexibility on supply function competition and market prices.\footnote{ Other related studies include \cite{wu2013curtailing, zhou2019managing, sunar2021net,aflaki2017strategic, hu2015capacity, murali2015municipal, kok2018impact, lobel2011consumer}, which  analyze the optimal renewable energy operations and investments.}
In contrast, our paper focuses on how renewable suppliers influence the LSE’s bidding behavior in the DA market in the presence of virtual traders.

On the methodological side, we contribute to the supply function equilibrium literature. \cite{klemperer1989supply} introduce a supply function competition model with demand uncertainty and convex production cost. \cite{johari2011parameterized} study the SFE in a deterministic context. They find that restricting the strategy space of participants to choose their supply functions can improve market efficiency.
\cite{green1992competition} and \cite{green1996increasing} apply SFE models in electricity markets, considering a deterministic but time dependent demand. 
\cite{holmberg2007supply,holmberg2008unique} studies supply function equilibrium with symmetric and asymmetric capacity constraints. 
Our paper differs from these studies in that we analyze SFE in a two
stage game, featuring the interaction of DA and RT markets. 
Table \ref{tab:sfe_lit_comprs} summarizes existing literature on supply function equilibrium models.
\begingroup
\begin{table}[h!]
  \centering
  \scriptsize
  \begin{tabular}{l||c|c|c|c|c|c}
    \hline
    Paper\textbackslash Feature & Two stage & General & Random & Renewable & Capacity & Time varying \\
                                & SFE & Function Form & Demand & Supplier & Constraint & Demand \\ \hline
    \cite{klemperer1989supply}                 & \xmark & \xmark & \cmark & \xmark & \xmark & \xmark \\
    \cite{johari2011parameterized}             & \xmark & \xmark & \xmark & \xmark & \xmark & \xmark \\
    \cite{holmberg2007supply, holmberg2008unique} & \xmark & \cmark &
                                                                      \cmark & \xmark & \cmark & \xmark \\ 
    \cite{sunar2019strategic}                  & \xmark & \cmark & \cmark & \cmark & \xmark & \xmark \\
    \cite{green1992competition} & \xmark & \cmark & \xmark & \xmark & \xmark & \cmark \\
    Our paper  & \cmark & \cmark & \cmark & \cmark & \cmark & \cmark \\ \hline 
\end{tabular}
\caption{
  Comparison with
  existing literature 
  on supply function equilibrium models.}
\label{tab:sfe_lit_comprs}
\end{table}
\endgroup


Lastly, our study is related to studies on energy commodity trading as explored by \citep{wu2005competitive, secomandi2014optimal,trivella2023meeting, dong2007equilibrium,mendelson2007strategic,pei2011sourcing}. Our unique contribution is the investigation of how virtual trading and  renewable energy suppliers influence price dynamics and DA demand of LSEs.\looseness=-1

The rest of this paper is organized as follows. We provide institutional details on the U.S. multi-settlement electricity market in \S\ref{sec:institdetails}. We develop the model in \S\ref{sc:Model} and construct the supply function equilibrium in  \S\ref{sc:Analysis}. We analyze the impact of virtual traders on the two settlement market in \S\ref{sc:Eff_VT}. We provide empirical support to our main theoretical predictions using historical ISO data in  \S\ref{sc:NumStudy}. We conclude in \S\ref{sc:Conclusion}. All derivations and proofs are delegated to appendices.


\section{Institutional Details} \label{sec:institdetails}
The multi-settlement electricity market in the U.S. typically operates as follows. In the DA market, which takes place one day prior to the RT market\footnote{The DA market covers an entire day, segmented into individual hours, whereas each execution of the RT market covers a shorter time segment; typically five minutes.}, participants submit their supply and demand bids. These bids, also known as supply and demand curves, specify the quantities and prices they are willing to trade. An ISO collects this data and sets market clearing prices by matching supply and demand. Hence, suppliers and LSEs effectively enter into forward contracts. Within these contracts, LSEs commit to consuming specified quantities of electricity, while suppliers commit to delivering these agreed-upon amounts at predetermined market prices for the next day. \looseness=-1

Between the DA and RT markets, suppliers can adjust their offers to reflect updated conditions, submitting new supply curves based on residual capacity—estimated after meeting DA commitments. This adjustment window is typically close to RT operations; for example, in California it closes 75 minutes before the RT market begins.\footnote{See \url{http://www.caiso.com/market/Pages/MarketProcesses.aspx} for details on California’s market procedures.}

The RT market begins after the adjustment period closes. At this stage,
actual demand and available capacities are realized, but participants must
honor their DA market commitments. If actual demand exceeds the DA-cleared
quantity, LSEs meet the shortfall by purchasing additional electricity in
the RT market. Conversely, renewable suppliers with surplus capacity can
offer extra supply in the RT market. Suppliers unable to fulfill their DA
obligations—due to capacity shortfalls—must procure power in the RT
market, potentially incurring penalties. The ISO allocates residual demand
among suppliers, curtails excess supply in the event of a shortfall, and
clears the RT market at prices determined by submitted curves and
contracted quantities from both the DA market and the adjustment period.

\section{Model}
\label{sc:Model}
We study a two-settlement electricity market with day-ahead (DA) and
real-time (RT) trading for consumption over a finite time interval $\mathcal{T}$. 
The market participants include an aggregated LSE, an aggregated conventional supplier, a set $\mcal{S}:= \{1,..., N_S\}$ of $N_S~(\ge 2)$ renewable suppliers, and a set $\mcal{V}:= \{1, \ldots, N_V\}$ of $N_V~(\geq2)$ virtual traders.
All market participants know the distribution of the demand and renewable capacities when submitting their bids in the DA market.
They subsequently refine their bids in the RT market after observing the realized demand and capacities. 
During the consumption period $\mathcal{T}$, physical exchange of electricity takes place based on the contracts entered in both the DA and RT markets.\looseness =-1  

Demand $\{\tilde{D}(t): t \in \bigT\}$ is assumed to be a bounded  continuous stochastic process over time~$t$, where $\tilde{D}(t) \in [\underline{D},\overline{D}]$ for known constants $\overline{D} > \underline{D}\geq 0$.
In the DA market, all participants only know the distribution of the stochastic process $\{\tilde{D}(t): t\in\mathcal{T}\}$ together with the upper and lower bounds.
The realized demand $\{D(t): t\in \mcal{T}\}$ becomes known before the participants bid in the RT market~\citep{green1992competition, green1996increasing}.

We assume the capacity $\{\tilde{Q}_i(t): t\in\mcal{T}\}$ to be a
stochastic process that is continuous over time $t$, (weakly) stationary,
and independent and identically distributed for all $i \in \mcal{S}$.  
The range of $\tilde{Q}_i(t)$ is $[0, \overline{Q}]$, where $\overline{Q}>0$ is a known constant.
As with demand, the participants only know the distributions of $\{\tilde{Q}_i(t): i \in \mcal{S}, t\in\mcal{T}\}$ when bidding in the DA market. 
The realized values ${Q_i(t): i \in \mathcal{S},, t \in \mathcal{T}}$ become known before bids are updated in the RT market.
We further assume that the conventional supplier faces no capacity constraints—a reasonable assumption given that, as noted earlier, blackouts are rare events
\citep[page 2]{CAISO_Outage}. 
\begin{figure}[h!]
\centering
\includegraphics[scale=0.14]{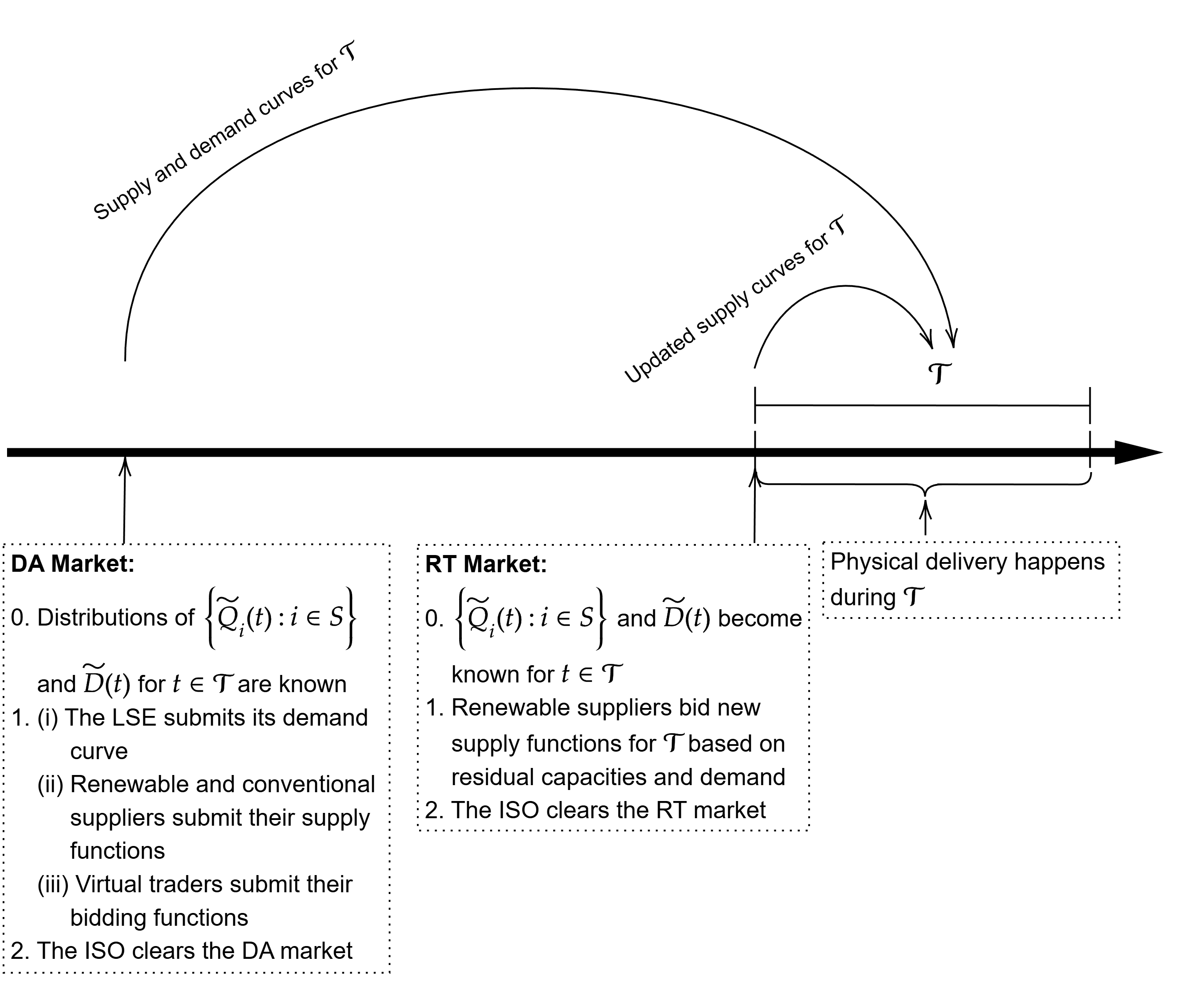}
\caption{Model Timeline}
\label{fig:timeline}
\end{figure}

Participants engage in the DA and RT markets as follows {(see
Figure~\ref{fig:timeline} for a schematic representation of the timeline)}.
In the DA market, the LSE submits a continuous demand curve $D_l :=\{D_l(t): t \in \mcal{T}\}$. 
This curve indicates the LSE's planned electricity purchase to meet the demand requirements throughout the time period $\mathcal{T}$.  
If $D(t) > D_{l}(t)$ for some $t \in \mcal{T}$, the LSE balances the shortfall by purchasing additional electricity from the RT market at the prevailing RT price. 
Both DA and RT demand from the LSE is treated as inelastic to reflect the common LSE practice of selling electricity to retailers at fixed prices through short-term contracts (see, e.g., \citealt{csereklyei2020price, burke2018price}).\footnote{We assume the LSE's demand to be inelastic, and therefore independent of prices, as commonly done in literature (see, e.g., \citealt{johari2011parameterized, holmberg2007supply, holmberg2009numerical,xu2015demand} and references therein). In practice, LSEs are allowed to submit elastic demand curves, which are functions of both price and time.}$^{,}$\footnote{Because of the inelastic nature of demand, each LSE's demand is expected to be fulfilled without competitive interactions among them, and thus we model the LSEs collectively as a single entity.} 
Renewable and conventional suppliers submit their supply functions in both markets.
Virtual traders only participate in the DA market by submitting their bidding functions. 
The ISO determines market clearing prices for the DA and RT markets
separately, ensuring that supply and demand are balanced within each
period.


\subsection{DA Market}\label{sec:DAMarket}
Let $S_i: \mathbb{R}^{+}\rightarrow\mathbb{R}^{+}$ denote the DA supply function of supplier
$i\in \bigS$, where $\mathbb{R}^+:=[0,+\infty)$. For a given DA price $p_f(t) \in \mathbb{R}^+$ at time $t \in \mathcal{T}$, $S_i(p_f(t))$ denotes the quantity of electricity that supplier $i$ is willing to provide.\footnote{We do not consider negative prices in this paper, although it may occur in real markets.} In our setting, the $i$-th renewable supplier submits a single supply function $S_i$ for the entire consumption period. 
The supply function starts at the minimum supply level mandated by ISO
guidelines (see, e.g., \citealp[page 59]{CAISO_BPM}). Given the negligible
production costs of renewable suppliers, we set $S(0) = 0$.
We also assume that $S_i$ is differentiable and monotonically increasing in price, consistent with standard assumptions in the literature \citep{klemperer1989supply,green1992competition,holmberg2010supply,sunar2019strategic}. 
In addition, since the capacity distribution is assumed to be symmetric across renewable suppliers, we posit the following equations for the supply functions:
\begin{equation}
\label{DA:Symmetry}
S_1=S_2=...=S_{N_S}=S.
\end{equation}

The supply function $C:\mathbb{R}^{+}\rightarrow\mathbb{R}^{+}$ of the conventional supplier is known to every participant, including the ISO (reflecting existing practices, e.g., \citealp[page 30]{NYISO_RLM}). 
We denote by  $\underline{p}_c$ the minimum price the conventional supplier is willing to supply, i.e., $C(p)=0$ for $p\in[0,\underline{p}_c]$. 
For $p>\underline{p}_c$, we assume that $C$ is differentiable, (weakly) concave, and (strictly) increasing in price, in line with existing literature.
The concavity of the supply function implies diminishing supply rates (or equivalently, an increasing marginal cost) of the conventional supplier as the price increases.\footnote{We do not consider the minimum production quantity of the conventional supplier. 
This quantity is typically set to reflect substantial operational costs associated with production resumption in practice. 
In recent years, the minimum production quantity has been decreasing (see, e.g., \citealt{denholm2011grid}, and \citealt[page 2]{CAISO_Curtail}) due to the
ISO's efforts to incorporate more clean energy into the grid.}

The virtual bidder $v \in \bigV$ participates in the DA market by submitting a differentiable and (strictly) decreasing bid function $B_v: \mathbb{R}^{+}\rightarrow\mathbb{R}^{+}$. We interpret $B_v$ as the quantity of electricity that the virtual trader is willing to buy (i.e., take a long position) at a given price when acting as a DEC bidder. For an INC bidder, we interpret $-B_v$ as the quantity the trader is willing to sell (i.e., take a short position) in the DA market. Virtual traders do not participate in the RT market; instead, their positions will be automatically cash settled at the RT price. We consider the net effect of virtual trading, assuming that all virtual traders are either INC or DEC bidders. 
For clarity of exposition, we model them as DEC bidders; analogous results hold when all virtual traders are INC bidders.\looseness=-1

The ISO clears the DA market by choosing a price $p_{f}^*(t)$ that equates the total supply from the renewable and conventional suppliers to the demand from the LSE and virtual DEC traders for time $t\in\mathcal{T}$, i.e.,\footnote{Similar market clearing conditions have been widely used in the literature (see the studies of \citealt{klemperer1989supply,green1992competition, anderson2002using, sunar2019strategic} and references therein).}
\begin{equation}
\label{DA:MarketClearing}
\sum_{i \in \bigS}S_i\big(p_{f}^*(t)\big)+C(p_{f}^*(t))=D_{l}(t) + \sum_{v\in\mathcal{V}}B_v\big(p_{f}^*(t)\big).
\end{equation}
The market clearing process sets up a forward contract among market participants. 
It gives the LSE the right (and the obligation) to receive $D_l(t)$ units of electricity at the predetermined price $p_{f}^*(t)$ to fulfill the demand $D(t)$ at time $t \in \bigT$. 
The virtual DEC bidders purchase $\sum_{v\in\mathcal{V}}B_v\big(p_{f}^*(t)\big)$ units at this price, anticipating to resell them at the RT price $p^*_{s}(t)$ for $t \in \bigT$.
Meanwhile, renewable suppliers (resp. conventional supplier) commit to delivering $\sum_{i\in\bigS}S_i(p_{f}^*(t))$ (resp. $C(p_{f}^*(t))$) units of electricity at time $t\in\mathcal{T}$.

\subsection{RT Market}
In the RT market, the realized values of the demand $D(t)$ and the
capacity $Q_{i}(t)$ for $t\in\mathcal{T}$ become known to all
participants.  
Let 
$D_r(t):=D(t)-\sum_{i\in\mathcal{S}}\min\{S_i\big(p_{f}^*(t)\big),Q_i(t)\}-C\big(p_{f}^*(t)\big)$
denote the residual demand for $t\in\mathcal{T}$, where the minimization
reflects the fact that the actual renewable supply is subject to the
capacity. Define $\underline{D}_r:=\min_{t\in\mathcal{T}}\{D_r(t)\}$ and
$\overline{D}_r:=\max_{t\in\mathcal{T}}\{D_{r}(t)\}$, respectively. 
Then based on these realized values, we need to consider the following two
different cases in the RT market.\looseness =-1 

\noindent{\textbf{Case 1: $D_{r}(t)\leq 0$}}
\label{Case1}\\
\noindent 
When the residual demand is non-positive, the LSE
does not need to purchase additional electricity from the RT market, as the supply reserved in the DA market exceeds the actual demand for time $t\in\mathcal{T}$. 
In such case, the ISO reduces the excess supply to align with the actual demand, typically by curtailing the conventional supplier first, given its  
higher production costs compared to renewable suppliers. 
The LSE is not eligible for refunds for any surplus reserved in the DA market since these funds have been used to prepare generator resources. Suppliers can produce less than their DA commitments without facing penalties, as the underproduction does not incur additional costs. 

\noindent{\textbf{Case 2:} $D_{r}(t)>0$}
\label{Case3}\\
\noindent
In this case, the LSE needs to purchase electricity from the RT market to
satisfy the positive residual demand for time $t\in\mathcal{T}$. Let
$\bar{S}_i: \mathbb{R}^+ \rightarrow \mathbb{R}^+$ denote the supply function of the $i$-th supplier in the RT
market. We interpret $\bar{S}_i(p)$ as the incremental supply above
the committed supply $\min\{S_i(p_f^*(t)),Q_i(t)\}$ at price $p \in\mathbb{R}^+$. {We need to ensure that 
\[
\bar{S}_i(p) + S_i\big(p_f^*(t)\big) \leq Q_i(t), \qquad t\in\bigT.
\]
Thus, the residual capacity 
$Q_{i}^r$ that can be
offered in the RT market is given by 
\begin{equation}
  \label{RT:ResidualCapacity}
  Q_{i}^r = \min\{\max\{Q_{i}(t) - S_i\big(p^{*}_{f}(t)\big),0\}: t \in \bigT\} =
  \min\{\max\{Q_{i}(t) - S\big(p^{*}_{f}(t)\big),0\}: t \in \bigT\},
\end{equation}
where the 
second equality in 
\eqref{RT:ResidualCapacity} holds due to \eqref{DA:Symmetry}.} {The outer minimization reflects that renewable suppliers must preserve sufficient residual capacity for every $t \in \mathcal{T}$, as they are restricted to bidding a single supply function in the real-time market.}


Since renewable suppliers have extra supply available in all periods $t\in
\bigT$, if and only if, their residual capacities are positive, we allow
the $i$-th renewable supplier to bid in the RT market only when $Q_i^r >
0$. Suppliers with zero residual capacities cannot be strategic in the RT market. Let
$\bigS_r \subseteq \bigS$ be 
the set of strategic suppliers, i.e., the set of renewable suppliers with positive residual capacities. 
A non-strategic supplier 
$j \in\bigS\setminus\bigS_r$  must still
fulfill its DA commitment $S_j\big(p_{f}^*(t)\big)$. If its realized
capacity $Q_{j}(t) <S_j\big(p_{f}^*(t)\big)$, the renewable supplier
incurs a shortfall penalty in that it must 
purchase the shortfall
$S_j\big(p_{f}^*(t)\big)- Q_{j}(t)$ from the RT market at the RT market
clearing price $p_{s}^*(t)$. \looseness = -1
We assume that $\bar{S}_{i}$ is continuous, monotonically increasing, and
piecewise differentiable until it reaches the capacity bound 
\begin{equation}
\label{RT:CapacityConstr}
\bar{S}_i(p)\leq Q_{i}^r,\ \ \forall p \in\mathbb{R}^+,
\end{equation}
{Thus, it follows that $\bar{S}_i(p) = Q_i^r$ for all $p \geq \inf\{y:
  \bar{S}_i(y) \geq Q^r_i\}$, i.e. 
on reaching the upper bound, $\bar{S}_i$ becomes constant.} 

{As in the
DA market, the ISO chooses a RT market clearing price $p_{s}^*(t)$ to balance the supply and the demand for $t\in\mathcal{T}$:
\[
  \sum_{i \in \bigS} \min\{S\big(p^*_f(t)\big), Q_i(t)\} + \sum_{i \in
    \bigS_r}\bar{S}_i\big(p_{s}^*(t)\big) + \max\{C(p_s^*(t)),
  C(p_f^*(t))\} = D(t),
\]
where the supply from the conventional supplier is set to $\max\{C(p_s^*(t)),
C(p_f^*(t))\}$. This is because the ISO is not allowed to curtail the
conventional supplier when the residual demand is positive.    
Rearranging terms, we get:\looseness= -1
\begin{equation}
\label{RT:MarketClearing}
\sum_{i \in \bigS_r}\bar{S}_i\big(p_{s}^*(t)\big)+\bar{C}\big(p_{s}^*(t)\big)  - C(p_f^*(t))
= D_r(t),
\end{equation}
where we use the notation  $\bar{C}(p_s^*(t)) := \max\{C(p_s^*(t)), C(p_f^*(t))\}$.
}

In Case 2, the ISO does not penalize renewable suppliers for
overproduction, i.e., produce more than their DA commitments, because the usage of clean energy would improve social
welfare, through reducing the pollution and greenhouse gas emission from
thermal plants using fossil fuels.  

\subsection{Participants' Decision Problems and Equilibrium Definition}
The objective of the $i$-th renewable supplier is to maximize its total profit by selecting a supply function $S_i$ in the DA market and then refining it to $\bar{S}_i$ in the RT market (if possible). 
The $v$-th virtual trader, restricted to a single bidding opportunity in the DA market, maximizes its profit by choosing the bidding function $B_v$ in the DA market.
The aggregated LSE minimizes costs by allocating its demand between the two markets: purchasing $D_{l}(t)$ in the DA market and $\max\{D(t)-D_{l}(t),0\}$ in
the RT market for each delivery time $t\in\mathcal{T}$.  
Our analysis of this two stage game employs backward induction, starting with the decision problem of the renewable supplier in the RT market, where it operates as the sole strategic player. 

In the RT market, renewable suppliers have knowledge about the decision outcomes in the DA market and the realized demand and capacities. 
Let $D := \{D(t): t \in \mathcal{T}\}$ denote the realized demand curve, and $\mathbb{Q}:=\{Q_i(t): t \in \bigT, i \in \bigS\}$, denote the collective capacity profile of all renewable suppliers. 
Recall that the set of strategic suppliers $\bigS_r$ in the RT market are those with positive residual supply. 
{Let $\bar{\mathbb{S}}_{-i}$ represent the RT supply functions of all strategic suppliers except the $i$-th one, and $\mathbb{S}$ denote the collection of DA market supply functions submitted by all renewable suppliers.}
The payoff $\bar{\Pi}^{\mathrm{S}}_{i,t}(\bar{S}_i;\bar{\mathbb{S}}_{-i},
\mathbb{S},\mathbb{Q},D)$ of the strategic supplier $i \in \bigS_r$ using a
feasible supply function $\bar{S}_i$ for time $t$ is given by
\begingroup
\allowdisplaybreaks
\begin{align}
\label{RT:RS_Obj}
\bar{\Pi}^{\mathrm{S}}_{i,t}(\bar{S}_i;\bar{\mathbb{S}}_{-i}, \mathbb{S},\mathbb{Q},D):=\bar{S}_i\big(p_{s}^*(t)\big)p_{s}^*(t), 
\end{align}
\endgroup
where $p_{s}^*(t)$ is the RT market clearing price determined by
\eqref{RT:MarketClearing}, and 
the notation
$\bar{\Pi}^{\mathrm{S}}_{i,t}(\bar{S}_i;\bar{\mathbb{S}}_{-i},
\mathbb{S},\mathbb{Q},D)$ emphasizes that $\bar{S}_i$ is chosen with the
knowledge of $\big(\bar{\mathbb{S}}_{-i},
\mathbb{S},\mathbb{Q},D\big)$. 
The SFE for the RT market is formally defined as follows.
\begin{definition}{\textbf{(SFE in the RT market)}}
\label{Def:RT_SFE}
Given the DA supply function profile $\mathbb{S}$ of renewable suppliers,
their collective capacity profile $\mathbb{Q}$, and the actual demand
curve $D$, a RT supply function profile
$\bar{\mathbb{S}}^*=(\bar{S}^*_1,...,\bar{S}_{N_S}^*)$ is an SFE if the
following conditions hold for $i\in\mathcal{S}_r$ and $t\in\mathcal{T}$: 
\begin{enumerate}[(i)]
\item $\bar{S}_i^*:\mathbb{R}^+\rightarrow\mathbb{R}^{+}$ is continuous,
  increasing,  satisfies $\bar{S}_i^*(0)=0$, and \eqref{RT:CapacityConstr}; 
\item
  $\bar{\Pi}^{\mathrm{S}}_{i,t}(\bar{S}^*_i; \bar{\mathbb{S}}_{-i},\mathbb{S},\mathbb{Q},D)
  \geq
  \bar{\Pi}^{\mathrm{S}}_{i,t}(\bar{S}_i;
  \bar{\mathbb{S}}_{-i},\mathbb{S},\mathbb{Q},D)$, 
  for any other feasible supply functions $\bar{S}_i$ of the renewable supplier $i$. 
\end{enumerate}
\end{definition}
The above equilibrium definition requires at least two strategic renewable suppliers in the RT
market. If only one supplier has positive residual capacity, the ISO would require it to bid at a competitive price—its marginal cost—given the absence of competition. In practice, the ISO may employ various tools to detect and mitigate the supplier's market power (see, e.g., \citealt{CAISO_MPMitigation}). 
We show the presence of such equilibrium in \S\ref{sc:Analysis}.

In the DA market, the $i$-th renewable supplier submits its DA supply function $S_i$ based on the DA supply function profile $\mathbb{S}_{-i}:=(S_1,...,S_{i-1},S_{i+1},...,S_{N_S})$ of other renewable suppliers, the DA demand curve $D_l$ of the LSE, the decision profile $\mathbb{B}:=(B_1,...,B_{N_V})$ of virtual traders, and the RT supply function profile $\bar{\mathbb{S}}^*$ of all renewable suppliers.\footnote{The optimal RT supply function profile $\bar{\mathbb{S}}^{*}(\mathbb{S},\mathbb{Q},D)$ is a function of the DA supply function, and the realized capacities and demands. 
For simplicity, we suppress the dependence in our notation.} The corresponding payoff for the delivery time $t\in\mathcal{T}$ is given by
\begin{align}
\label{DA:RS_Obj}
\Pi_{i,t}^{\mathrm{S}}(S_i;D_l,
  \mathbb{S}_{-i},\mathbb{B},\bar{\mathbb{S}}^*) :=
  S_{i}\big(p_{f}^*(t)\big)p_{f}^*(t) +
  \mathbb{E}\bigg[\bar{S}^*_i\big(p_{s}^*(t)\big)p_{s}^*(t) + \bigg(Q_i(t)-S_i\big(p_f^*(t)\big)\bigg)^{-}p_s^*(t)\bigg],
\end{align} 
where $p_{f}^*(t)$ denotes the DA market clearing price determined by \eqref{DA:MarketClearing}, and the expectation is taken w.r.t. demand and capacities. The notation $(\cdot)^{-}$ refers to the negative part of the expression. 
The $i$-th renewable supplier's payoff consists of two parts: the immediate payoff from the DA market and the expected payoff from the RT market. The second term in the expectation indicates the shortfall penalty the supplier must pay when the realized capacity is less than its DA commitment.

The LSE bids a DA demand curve $D_l$ given the DA supply function profile of renewable suppliers $\mathbb{S}$, the DA bidding function profile
$\mathbb{B}$ of virtual traders, and the RT supply function profile of renewable suppliers $\bar{\mathbb{S}}^*$.
The payoff of the LSE for $t \in \bigT$ is expressed as
\begin{align}
\Pi^{\mathrm{L}}_t(D_{l};\mathbb{S},\mathbb{B},\bar{\mathbb{S}}^*):=p_{f}^*(t)D_{l}(t)+
\mathbb{E}\Big[\big(D(t)-D_{l}(t)\big)^{+}p_{s}^*(t)\Big],\label{DA:LSE_Obj}
\end{align}
where $(\cdot)^+$ refers to the positive part of the term in parenthesis, and the expectation is taken
w.r.t. demand and capacities. In \eqref{DA:LSE_Obj}, the first term is the
cost associated with purchasing $D_l$ in the DA market, while the second term is the
expected cost for purchasing the residual demand in the RT market.\looseness = -1

The $v$-th virtual trader submits its bidding function $B_v$ with the knowledge of the renewable suppliers' DA decision profile $\mathbb{S}$, the LSE's DA demand curve $D_l$, the bidding function profile $\mathbb{B}_{-v}:=(B_1,...,B_{v-1},B_{v+1},...,B_{N_V})$ of the other virtual traders in the DA market, and the RT supply function profile of renewable suppliers $\bar{\mathbb{S}}^*$. The payoff of the $v$-th virtual trader for $t\in\mathcal{T}$ is given by
\begin{equation}
\label{DA:VB_Obj}
\Pi^{\mathrm{V}}_t(B_v;\mathbb{B}_{-v}, \mathbb{S}, D_l,\bar{\mathbb{S}}^*):=\Big(\mathbb{E}\big[p_{s}^*(t)\big]- p_{f}^*(t)\Big)B_v\big(p_{f}^*(t)\big),
\end{equation}
where the expectation is taken w.r.t. demand and capacities. The DA market equilibrium is defined as follows.
\begin{definition}{\textbf{(Equilibrium in the DA Market)}}
\label{Def:DA_Equilibrium}
Given the RT equilibrium solution $\bar{\mathbb{S}}^*$, the decision
profile $(\mathbb{S}^*,D_l^*,\mathbb{B}^*)$ is a DA market equilibrium if the following conditions hold:
\begin{enumerate}[(i)]
\item The function $S^*_i:\mathbb{R}^{+}\rightarrow\mathbb{R}^{+}$ is continuous and increasing, and satisfies $S_i^*(0)=0$ and  
\begin{equation}
\Pi_{i,t}^{\mathrm{S}}(S_i^*;\mathbb{S}_{-i}^*,\mathbb{B}^*,\bar{\mathbb{S}}^*,
D_l^*)\geq
\Pi_{i,t}^{\mathrm{S}}(S_i;\mathbb{S}_{-i}^*,\mathbb{B}^*,\bar{\mathbb{S}}^*,
D_l^*), \label{Def:RT_Equilibrium_1} 
\end{equation}
for all $t \in \bigT$, and all feasible DA supply functions $S_i$ for
supplier $i \in \bigS$.
\item The function $B^*_v:\mathbb{R}^{+}\rightarrow\mathbb{R}^{+}$ is
  continuous and decreasing, and satisfies 
\begin{equation}
\Pi_t^{\mathrm{V}}(B_v^*;\mathbb{B}_{-v}^*, \mathbb{S}^*, D_l^*)\geq
\Pi_t^{\mathrm{V}}(B_v;\mathbb{B}_{-v}^*, \mathbb{S}^*, D_l^*),  \label{Def:RT_Equilibrium_3}    
\end{equation}
for all $t \in \bigT$, and feasible bidding functions for virtual bidder $v
\in \bigV$. 
\item $\Pi_t^{\mathrm{L}}(D_l^*;\mathbb{S}^*,\mathbb{B}^*)\leq
  \Pi_t^{\mathrm{L}}(D_l;\mathbb{S}^*,\mathbb{B}^*)\label{Def:RT_Equilibrium_2}$,
  for all times $t \in \bigT$.
\end{enumerate}
\end{definition}
One of the main results of this paper is to show that such equilibrium exists.

\section{Equilibrium Analysis}
\label{sc:Analysis}
This section analyzes the model presented in \S\ref{sc:Model} using backward induction.  
We establish the equilibrium results for the RT market in \S\ref{ssc:RTAnalysis} and derive the DA equilibrium in \S\ref{ssc:DAAnalysis}.
\subsection{RT Market}
\label{ssc:RTAnalysis}
In the RT market, we consider different cases based on the realized values
of demand and capacities as discussed in \S\ref{sc:Model}. In Case 1, since $D_r(t)\leq 0$ for all
$t\in\mathcal{T}$, i.e., the LSE's DA reservation is sufficient to cover
its actual demand, there is no RT market activity. Consequently, our
analysis primarily focuses on Case 2. 

In Case 2, the residual demand 
$D_r(t)>0$ for some $t\in\mathcal{T}$. To
streamline our discussion, we start from a 
scenario 
that satisfies the following assumption.
\looseness = -1 
\begin{assumption}
\label{Ass:RT}
The demand and capacity realizations satisfy the following conditions:
\begin{enumerate}[(i)]
\item The minimal residual demand is zero, i.e., $\underline{D}_r=0$.
\item All renewable suppliers have sufficient realized capacities to fulfill their
  DA commitments, i.e., the residual capacity $Q_{i}^r>
  0$ for all $i\in\mathcal{S}$, and therefore, $\mathcal{S}=\mathcal{S}_r$. 
\item All strategic renewable suppliers can reach their residual
  capacities at some $t\in\mathcal{T}$. 
\end{enumerate}
\end{assumption}
We 
relax Assumption~\ref{Ass:RT} in Appendix \ref{sec:robustmodelass}.
From the symmetry condition~\eqref{DA:Symmetry} and the definition
\eqref{RT:ResidualCapacity} of  residual capacity, it follows, without
loss of generality, that the residual capacities of renewable suppliers
can be ordered in a (strictly) ascending order, i.e.,  $0:=Q_0^r < Q_1^r
< \ldots <Q_{N_S}$, where the first inequality follows from 
Assumption 1~(iii). 
Based on this observation, the following lemma shows that in the SFE, a
renewable supplier's RT supply function is either the same as other
renewable suppliers or equals its residual capacity limit.\looseness = -1 

\begin{restatable}[\textbf{Symmetric supply functions}]{lemma}{RTSymmetry}
\label{RT:Symmetry}
Under Assumption \ref{Ass:RT}, any 
SFE in the RT market 
is of the form
\begin{equation}
  \label{eq:RT-symmetric-SFE}
  \bar{S}_i^*(p)=\min\{\bar{S}(p),Q_{i}^r\},\ \ \forall p\in\mathbb{R}^+, 
\end{equation}
where $\bar{S}:\mathbb{R}^+\rightarrow\mathbb{R}^+$ is a continuous and
increasing function.
\end{restatable}
This lemma states that the supply functions
$\bar{S}^*_i$ for all $i\in\mathcal{S}$ are identical before they reach their residual
capacities $Q^r_i$ for $i\in\mathcal{S}$. 

Next, we characterize the function $\bar{S}$ that defines the SFE in the
RT market.
{It follows from Assumption~\ref{Ass:RT} (all residual capacities are attainable) and Lemma \ref{RT:Symmetry} ($\bar{S}$ is a continuous increasing function) that there exists a RT market clearing price $p_s^*(t)$ that satisfies $Q^r_{k-1} < \bar{S}(p^*_s(t)) < Q^r_k$ for some $t \in \bigT$ and $k \in \{1,...,N_S-1\}$. } 
In this case, the first $k-1$ suppliers produce at their residual capacities, while the remaining $N_S - k + 1$ suppliers supply according to $\bar{S}$.\footnote{
This configuration arises because suppliers are ordered by increasing residual capacities $Q^r_i$.}
Then using the market clearing condition \eqref{RT:MarketClearing} and 
Lemma~\ref{RT:Symmetry}, the payoff \eqref{RT:RS_Obj} of the $k$-th
renewable supplier may be rewritten as follows:
\begin{align}
\label{RT:Payoff}
\bigg(D_r(t) - \big(\bar{C}(p_s^*(t)) - C(p_f^*(t))\big) -(N_S - k)\bar{S}(p_s^*(t))-\sum_{i=0}^{k-1}Q_{i}^r\bigg)p_s^*(t),
\end{align}
The equilibrium supply function $\bar{S}$ must satisfy the first order
optimality condition~(FOC)
\begin{equation*}
  D_r(t) - \big(\bar{C}(p_s^*(t)) - C(p_f^*(t))\big) - (N_S - k)\bar{S}(p_s^*(t))-\sum_{i=0}^{k-1}Q_{i}^r + p_s^*(t)\bigg(-\bar{C}'(p_s^*(t))-(N_S-k)\bar{S}^{\prime}(p_s^*(t))\bigg)=0.
\end{equation*}
Using again the market clearing condition \eqref{RT:MarketClearing} and
rearranging, 
the 
FOC can be rewritten as
\begin{equation}
\label{RT:ODE}
\big(p_s^*(t)(N_S-k)\big)\bar{S}^{\prime}(p_s^*(t))-\bar{S}(p_s^*(t))=-p_s^*(t)\bar{C}'(p_s^*(t)).
\end{equation}
We can view \eqref{RT:ODE} 
as an ordinary differential equation (ODE) for 
$\bar{S}$ in terms of the 
independent variable $p_s^*(t)$ that satisfies
$Q^{r}_{k-1} < \bar{S}(p_s^*(t)) < Q^r_{k}$ for $t\in\mathcal{T}$. 
The entire supply function is then constructed by sequentially solving \eqref{RT:ODE} for $k=1,...,N_S-1$, and ``stitching" these solutions together.
Specifically, for $k=1$, the solution to \eqref{RT:ODE} is 
\begin{equation}
\label{RT:Sol_0}
\bar{S}(p;\bar{c}_1)=\begin{cases}
0, & \text{if}\ p=0, \\
p^{\frac{1}{N_S-1}}\bigg(\bar{c}_1-\frac{1}{N_S-1}F_{1}(p)\bigg),
& \text{if}\ p\in(0,\gamma_1(\bar{c}_1)],\text{\footnotemark} 
\end{cases}
\end{equation}
\footnotetext{Since there is no need to emphasize the market clearing price $p_s^*(t)$ for time $t\in\mathcal{T}$ in the ODE solutions, we replace it with a generic variable $p\in\mathbb{R}^+$ to ease the notation.}
where $\bar{c}_1$ is a constant, $F_1$ is an antiderivative of
$p^{-1/(N_S-1)}\bar{C}'(p)$ (see detailed derivations in Online
Appendix \ref{sc:AppSolODE}), and  
\begin{equation*}
\gamma_1(\bar{c}_1):=\min\{p:\bar{S}(p;\bar{c}_1)=Q_1^r\}.
\end{equation*}
For $k\in\mathcal{S}\setminus\{1, N_S\}$, the solutions to \eqref{RT:ODE}
can be written as
\begin{equation}
\label{RT:Sol_1}
\bar{S}(p;\bar{c}_k)=
p^{\frac{1}{N_S-k}}\bigg(\bar{c}_k-\frac{1}{N_S-k}F_{k}(p)\bigg),\ \ p\in\big(\gamma_{k-1}(\bar{c}_{k-1}), \gamma_{k}(\bar{c}_{k})\big],
\end{equation}
where $\bar{c}_k$ is a constant, $F_{k}(p)$ an antiderivative of $\bar{C}'(p)\big(p\big)^{-1/(N_S-k)}$, and 
\begin{equation*}
\gamma_k(\bar{c}_k):=\min\{p:S(p;\bar{c}_k)=Q_k^r\}.
\end{equation*}
For any value of 
$\bar{c}_k$, equations~\eqref{RT:Sol_0}--\eqref{RT:Sol_1} define a class of SFE 
candidates that satisfy the FOC 
over a corresponding interval.
However, 
we also need to ensure that the 
function $\bar{S}$ is continuous and monotonically increasing. 
The following proposition establishes the existence of constants that
ensure continuity and monotonicity of $\bar{S}$.

\begin{restatable}[\textbf{Structure of RT market SFEs}]{proposition}{RTexistence}
\label{RT:Existence}
A function $\bar{S}(p)$ defines a RT market SFE candidate (as characterized in~\eqref{eq:RT-symmetric-SFE}) if, and only if, there exist 
constants
$\bar{\bm{c}} = (\bar{c}_1,\ldots, \bar{c}_{N_S})$ 
such that the function $\bar{S}(p;\bm{c})$ defined sequentially  via
\eqref{RT:Sol_0}--\eqref{RT:Sol_1}, for $k = 1, \ldots, 
N_S-1$, satisfy the following conditions: 
\begin{enumerate}[(i)]
\item The prices
  $\gamma_k(\bar{c}_k):=\min\{p:\bar{S}(p;\bar{\bm{c}})=Q_k^r\}$ are in an
  increasing order, i.e., $0:=\gamma_0 < \gamma_1(\bar{c}_1) <\ldots
  <\gamma_{N_S}(\bar{c}_{N_S})$.
\item $\bar{S}(p;\bar{\bm{c}})$ is 
  continuous and monotonically increasing for  $p > 0$, and satisfies
  $\lim_{p\to 0^+}\bar{S}(p;\bar{\bm{c}}) = 0$ and $\lim_{p\to
    0^+}\bar{S}'(p;\bar{\bm{c}}) = 0$.
\end{enumerate}
Moreover, the set of constants $\bar{\bm{c}}$ satisfying $(i)$ and $(ii)$
is non-empty, and the continuity requirement implies that all the constants
$\bar{c}_k$, $k \geq 2$, are increasing functions of $\bar{c}_1$, and therefore,
$\bar{S}(p;\bar{\bm{c}})$ reduces to a $1$-dimensional manifold. 
\end{restatable}
{The detailed proof of Proposition \ref{RT:Existence} is in the Appendix; we now describe its main structure. The sufficiency (``if") direction follows from Lemma~\ref{RT:Symmetry}. That is, a supply function equilibrium (SFE) should be of the form \eqref{RT:Sol_0}-\eqref{RT:Sol_1}, with constants $\bar{\bm{c}}$ chosen to ensure continuity and monotonicity. For the necessity (``only if"), we first show that
$\bar{S}$ is continuous and increasing within each interval
$(\gamma_{k-1},\gamma_k)$ if $\bar{c}_k$ exceeds a certain threshold. We then prove that the
continuity at each breakpoint $\gamma_k$ implies that $\bar{c}_k$, for each $k\in\mathcal{S}\setminus\{1,N_S\}$, is an increasing function of $\bar{c}_1$. 
The monotonicity is achieved by sufficiently increasing $\bar{c}_1$ to ensure all constants exceed their respective thresholds. 
It is worth noting that larger constants imply ``steeper" supply functions
and therefore, lower RT prices, leading to uniqueness, as formalized in the  following proposition.}
\begin{restatable}[\textbf{Uniqueness of the RT market SFE}]{proposition}{RTUniqueness}
\label{RT:Uniqueness}
Within the class of SFE candidates 
identified in 
Proposition
\ref{RT:Existence}, there exists a unique SFE for the RT market.
\end{restatable}
The intuition behind Proposition~\ref{RT:Uniqueness} is as follows. A ``flatter'' supply function means that the quantity offered changes little as the price increases.  Renewable suppliers prefer such flatter supply functions because they yield a higher market-clearing price for the same quantity of electricity supplied. Consequently, the SFE corresponds to the smallest value of the constant $\bar{c}_1$ that still produces a valid $\bar{S}$ within the candidate set of SFEs.


Appendix~\ref{sec:robustmodelass} analyzes variants of Case 2 by relaxing
Assumption \ref{Ass:RT} and shows that the main results and insights are
robust against those assumptions. Specifically, for a (strictly) positive
minimum residual demand $\underline{D}_r$, the $k$-th {renewable supplier adopts one of the two bidding strategies: (i) symmetric supply curves originating from a unique price–quantity pair $(p_k, \bar{S}(p_k))$, with $p_k>0$ and $\bar{S}(p_k)$ strictly below the residual capacity; or
(ii) full utilization of residual capacity throughout the RT market when $\bar{S}(p_k)$ exceeds the residual capacity.}
Moreover, if the maximum residual demand $\overline{D}_r$ can
be met without reaching $Q_{N_S}^r$, renewable suppliers would stop
increasing $\bar{c}_1$ when the total supply satisfies $\overline{D}_r$.\looseness = -1 

\subsection{DA Market}
\label{ssc:DAAnalysis}
In the DA market, the strategic participants include renewable suppliers, the LSE, and the virtual traders. In this section, we analyze a baseline model featuring  strategic renewable suppliers and the LSE only. The impacts of virtual trading are considered in the next section.

Renewable suppliers bid symmetric supply functions as stated in \eqref{RT:Symmetry}. For the $i$-th renewable supplier, we can rewrite its payoff \eqref{DA:RS_Obj} using the DA market clearing condition \eqref{DA:MarketClearing} as follows:
\begin{equation}
\label{DA:RS_OptObj}
\bigg(D_{l}-C(p_{f}^*)-S_{-i}(p_{f}^*)\bigg)p_{f}^*+\mathbb{E}\bigg[\bar{S}^*_i(p_{s}^*)p_{s}^*+ \big(Q_i-S_i(p_f^*)\big)^{-}p_s^*\bigg],
\end{equation}
where 
\begin{equation}
S_{-i}(x):=\sum_{j\in\mathcal{S}\setminus\{i\}}S_j(x)
\label{eq:supplmini}
\end{equation}
represents the total supply at price $x$ from all renewable suppliers except for the $i$-th one. As with the RT market, we suppress the time dependence in $D_l$, $p_f^*$, and $p_s^*$. 
Similarly, the LSE's payoff \eqref{DA:LSE_Obj} can be rewritten as follows using the DA market clearing condition \eqref{DA:MarketClearing}:
\begin{equation}
\label{DA:LSE_OptObj}
\bigg(\sum_{i\in\mathcal{S}}S_i(p_f^*)+C(p_f^*)\bigg)p_f^*+\mathbb{E}\bigg[\bigg(D-\sum_{i\in\mathcal{S}}S_i(p_f^*)-C(p_f^*)\bigg)^{+}p_s^*\bigg].
\end{equation}
The equilibrium supply function in the DA market is expected to satisfy the FOCs for \eqref{DA:RS_OptObj} and
\eqref{DA:LSE_OptObj} with respect to $p^*_f$, while also being continuous and monotonically
increasing. These FOCs result in ODEs for $S_i$ with respect to the independent variable $p^*_f$. The following proposition establishes the existence and
uniqueness of the SFE.\looseness = -1 
\begin{restatable}[\textbf{Existence and uniqueness of SFE in the DA
    market without virtual trading}]{proposition}{DAExistence} 
  \label{DA:Existence}
  Consider a market consisting of renewable and conventional suppliers and the LSE. Then there exists a unique SFE for the renewable suppliers in the DA market.
\end{restatable}

The existence result follows from identifying a common solution to the FOCs for both the renewable suppliers and the LSE. The proof of uniqueness parallels that of Proposition~\ref{RT:Uniqueness}.
To ensure the supply function is continuous and increasing, the constant in the solution must be sufficiently large. However, as in the RT market, renewable suppliers maximize their payoffs by adopting relatively flat supply functions—corresponding to small constants. The equilibrium therefore reflects a trade-off, with suppliers selecting the smallest constant that still guarantees continuity and monotonicity.


Based on the equilibrium decisions of 
renewable suppliers, 
the LSE can anticipate its procurement
cost for a given demand curve $D_l$ in the DA market. Therefore, the LSE chooses
a
DA demand curve $D_l^*$ that minimizes its total cost. 
However, the 
optimal demand curve $D^*_l$ may not be unique, leading to multiple DA equilibria. Note that the above
results are established in markets where virtual trading is absent. In the
next section, we will see that the optimal demand curve $D^*_l$ 
is
unique when 
there is perfect 
competition among virtual traders.  
\section{Impact of Virtual Trading on the Two Settlement Market}
\label{sc:Eff_VT}
We examine two scenarios: with and without renewable suppliers.\footnote{Virtual trading was introduced into the wholesale electricity market before the widespread adoption of renewable power.} We first show that the LSE can exercise market power when the conventional supplier is nonstrategic, and that virtual trading mitigates this power. We then introduce strategic renewable suppliers and analyze how their presence impacts the equilibrium strategy of both the LSE and virtual traders.

In the 
scenario 
without 
renewable
suppliers, the demand from the LSE is solely fulfilled by the nonstrategic
conventional supplier, i.e. the  model 
is identical to the one
in \S\ref{sc:Model} 
with 
all renewable supply functions 
set to
zero. 
\looseness=-1  
\begin{restatable}{proposition}{DAPriceGapVanilla}
\label{DA:PriceGap_Vanilla}
In a two-settlement electricity market comprising only the LSE and a conventional supplier, the equilibrium features strategic underbidding by the LSE, yielding
\[
p_f^* < \mathbb{E}[p_s^*].
\]\looseness = -1
\end{restatable}
Proposition \ref{DA:PriceGap_Vanilla} highlights the LSE's market power in
a two settlement electricity market 
without 
renewable
suppliers and virtual traders. In this setting, the LSE can depress the DA market clearing price by bidding below its expected RT demand. This strategic underbidding, which deviates from the truthful bidding expected by the ISO, allows the LSE to secure part of its demand at prices below the marginal cost implied by the conventional supplier’s supply function. In effect, the LSE’s cost savings come directly at the expense of the conventional supplier.
\begin{figure}[h!]
\centering
\includegraphics[scale=0.9, trim=0 0 50 0]{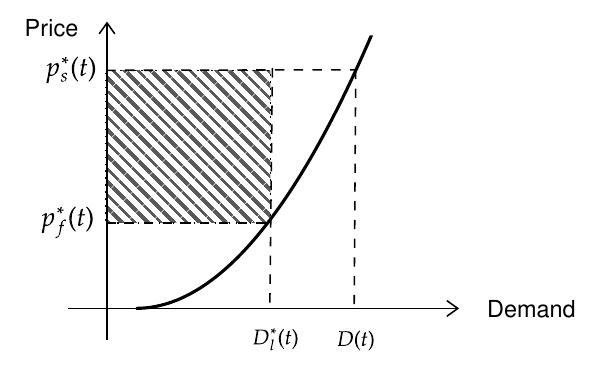}
\caption{The LSE's underbidding behavior when renewable suppliers and virtual traders are absent.}
\label{fig:MarketPowerLSE}
\end{figure}

Figure \ref{fig:MarketPowerLSE} 
explains the rationale 
for
Proposition~\ref{DA:PriceGap_Vanilla}. It illustrates the ex-post payoff of an LSE with positive residual demand. 
The LSE’s total payment in Figure~\ref{fig:MarketPowerLSE} equals the sum of its expenditures in the DA and RT markets,  
\[
p_f^*(t) D_l^*(t) + p_s^*(t) \big(D(t) - D_l^*(t)\big),
\]
where \(D(t) > D_l^*(t)\). 
This amount is strictly less than \(p_s^*(t) D(t)\), which the LSE would pay if it procured its entire demand \(D(t)\) directly from the conventional supplier. 
The difference---shown as the shaded rectangular area---represents the LSE’s savings, or equivalently, the conventional supplier’s loss. 
Thus, the LSE profits from exercising market power. 
In practice, such underbidding causes a price gap between DA and RT markets, which is highly detrimental to the market’s long-term efficiency \citep[page~26]{pjm2015virtual}.\looseness =-1 

When $p_f^*<\mathbb{E}[p_s^*]$, virtual traders are incentivized to submit
DEC bids to take advantage of the arbitrage opportunity between the DA and
RT markets, i.e., purchase electricity at the DA price and cash settle the purchased amount
at the RT price with an expected profit margin
$\mathbb{E}[p_s^*]-p_f^* > 0$.
The DEC bids are virtual loads, and they 
push up the DA market clearing price, and pull down the expected price
in the RT market. As a result, the profit margin for virtual traders,
alongside the marginal saving for the LSE, is reduced. In the limit of on
infinite number of traders,  
virtual trading 
aligns these two prices, as shown by the
following 
result.
\begin{restatable}{proposition}{DAPriceAlignmentVanilla}
\label{DA:PriceAlignment_Vanilla}
In a two settlement electricity market with the LSE, the 
conventional supplier, and an infinite number of virtual traders, the DA
market clearing price equals the expected RT market clearing price, i.e.,  $p_f^*=\mathbb{E}[p_s^*]$.
\end{restatable}
The price convergence is driven by perfect competition among an infinite number of virtual traders.
With only finitely many traders, the price gap would persist, as each trader balances the size of virtual positions against the price differential to maximize the product of the two.
Once DA and RT prices align, both the marginal profit of virtual traders and the marginal savings of the LSE fall to zero.

Next, we introduce
renewable suppliers 
into the market.
The following 
result
shows that the LSE maintains its market
power 
when 
virtual trading is not allowed.\looseness =-1 
\begin{restatable}{proposition}{DAPriceGap}
\label{DA:PriceGap}
In a two-settlement electricity market with the LSE, renewable suppliers, and the conventional supplier---but no virtual bidders---the LSE strategically underbids in the DA market at equilibrium, driving the DA clearing price below the expected RT clearing price, i.e., \( p_f^* < \mathbb{E}[p_s^*] \).\end{restatable}
Proposition \ref{DA:PriceGap} indicates that strategic renewable suppliers
cannot mitigate the LSE's market power that stems from the two settlement
structure. 
As in the case without renewable suppliers, 
the LSE 
is able to push down
the DA market
clearing price through underbidding, which results in  cost savings at the
expense of the renewable and conventional suppliers. 
The following 
result establishes
that, 
in the limit of an
infinite number of virtual
bidders, 
%
the
DA market clearing price converges to the expected RT market clearing, 
as in the case without
renewable suppliers. 
\begin{restatable}{proposition}{DAPriceAlignmentOne}
\label{DA:PriceAlignment_1}
In a two settlement electricity market that includes the LSE, the
renewable and conventional suppliers, and an infinite number of identical virtual
traders, the DA market clearing price equals the expected RT market
clearing, i.e., $p_f^*=\mathbb{E}[p_s^*]$. 
\end{restatable}
The next result shows that, in the limit of infinitely many identical virtual traders, the DA market admits a unique SFE.
\begin{restatable}[\textbf{Existence and uniqueness of SFE in the DA market with virtual trading}]{proposition}{DAExistenceTwo}
\label{DA:Existence2}
  Consider a market consisting of renewable and conventional suppliers and the LSE. Then there exists a unique SFE for the renewable suppliers in the DA market.  Moreover, the
form of the equilibrium supply function for renewable suppliers is 
identical to that 
in the case without virtual trading. 
\end{restatable}
Proposition~\ref{DA:Existence2} 
is the 
counterpart to Proposition~\ref{DA:Existence} 
for a market with
virtual trading. 
Interestingly, despite the inclusion of virtual trading, 
the renewable supplier's equilibrium supply function remains identical
in both scenarios 
This arises from the fact that, when prices are aligned, the influence of virtual trading on the equilibrium supply function cancels out entirely.

As mentioned in \S\ref{ssc:DAAnalysis}, there may be multiple optimal
bidding quantities for the LSE in the DA market without price
alignment. However, the following result establishes that LSE's
optimal demand in the DA market is unique under the price
alignment condition. 
\begin{restatable}{proposition}{DAUniquenessDemand}
\label{DA:Uniqueness_Demand}
Suppose there are both renewable and conventional suppliers and an
infinite number of virtual traders in the electricity market, then the LSE has a unique optimal DA demand bid. \looseness = -1
\end{restatable}
\section{Empirical Evidence on the Impact of Virtual Trading}
\label{sc:NumStudy}
We provide empirical support for the two main
model implications: (i) virtual trading narrows the price gap between the
DA and RT markets, and (ii) virtual trading  leads to a reduction in the
LSE's demand in the DA market. 

Using data from the California ISO (CAISO) and New York ISO (NYISO)—both of which offer observations from periods before and after the introduction of virtual trading—we are able to directly assess its market impact.\footnote{Other wholesale markets, such as the Electric Reliability
  Council of Texas (ERCOT), ISO New England (ISO-NE), Midcontinent
  Independent System Operator (MISO), and Pennsylvania-New Jersey-Maryland
  Interconnection (PJM), had virtual trading integrated since their
  opening (see, e.g., \citealt{parsons2015financial}), and thus do not
  support this comparative analysis.} 

We begin by examining the impact of virtual trading on  the difference
between the RT and DA prices in the California and New York electricity
markets. For CAISO, which implemented virtual bidding on February 1, 2011, we use daily data from April 2009 to March 2013, comparing the average hourly price gap across two periods: April 2009–January 2011 (pre–virtual trading) and February 2011–March 2013 (post–virtual trading). For NYISO, we analyze data from September 2001 (the earliest available on the NYISO website) through January 2002, capturing the period immediately before and after virtual trading was introduced on November 1, 2001. 
\begin{figure}[h!]
  \centering
  \begin{subfigure}[b]{0.44\textwidth}
    \includegraphics[width=\textwidth]{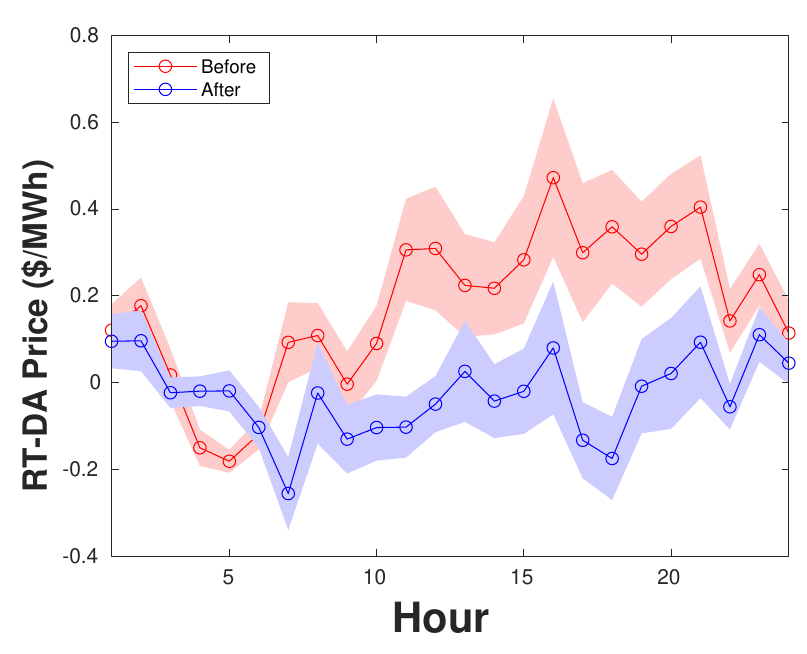}
    \caption{Price spreads for CAISO}
  \end{subfigure}
  \hfill
  \begin{subfigure}[b]{0.44\textwidth}
    \includegraphics[width=\textwidth]{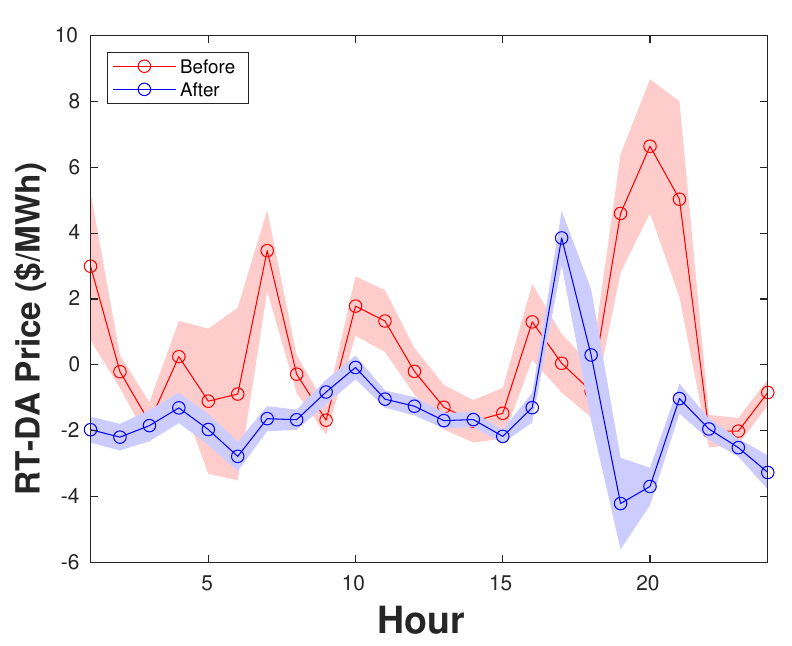}
    \caption{Price spreads for NYISO}
  \end{subfigure}
\caption{Average hourly price spreads, computed as the difference between RT and DA prices, before and after the introduction of virtual bidding.}
\label{fig:AverageSpreadHourly}
\end{figure}

Figure~\ref{fig:AverageSpreadHourly} plots average price gaps for each
hour of the day in both the CAISO and NYISO markets. While NYISO and CAISO
calculate the DA 
price on an hourly basis, they clear the RT market every five minutes. We
set the RT hourly price to the average of the twelve RT prices recorded
within that hour. 
In CAISO, before the introduction of virtual
trading, the average price gaps are positive throughout the day, except at
4 am and 5 am, indicating that the DA market clearing price is lower than
the RT price. This observation is largely consistent with our
model prediction of LSE's underbidding behavior in the DA
market. Implementing virtual trading reduces price gaps in most hours of
the day. In NYISO, DA prices were lower than RT prices prior to the introduction of virtual trading, but became higher afterward.\footnote{
A plausible explanation is that virtual bidders take virtual load positions to profit from extreme right-tail events in RT markets, such as price spikes caused by transmission congestion or reserve scarcity. By purchasing electricity in the DA market, they increase DA demand and prices; by closing these positions in the RT market, they increase RT supply and lower RT prices. These trading actions, on average, produce a negative price spread.} However, the average absolute price spread between the two markets decreased following the implementation of virtual bidding. 

Next, we compute the ratio of demand cleared in the DA market to that cleared in the RT market. We use this ratio 
as a surrogate for the 
extent to which the LSE
underbids. We compute average demand ratios for each hour in both
markets. 
As in the case for the 
market clearing price, we set the hourly demand in the 
RT market to the average demand over 12 
values recorded within each hour.\looseness = -1 
\begin{figure}[h!]
  \centering
  \begin{subfigure}[b]{0.43\textwidth}
    \includegraphics[width=\textwidth]{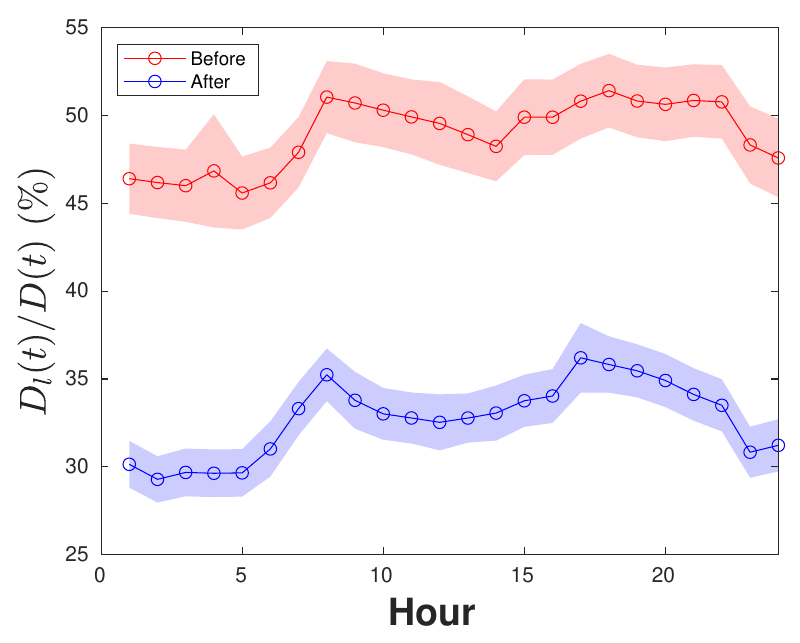}
    \caption{Demand Ratios for CAISO}
  \end{subfigure}
  \hfill
  \begin{subfigure}[b]{0.45\textwidth}
    \includegraphics[width=\textwidth]{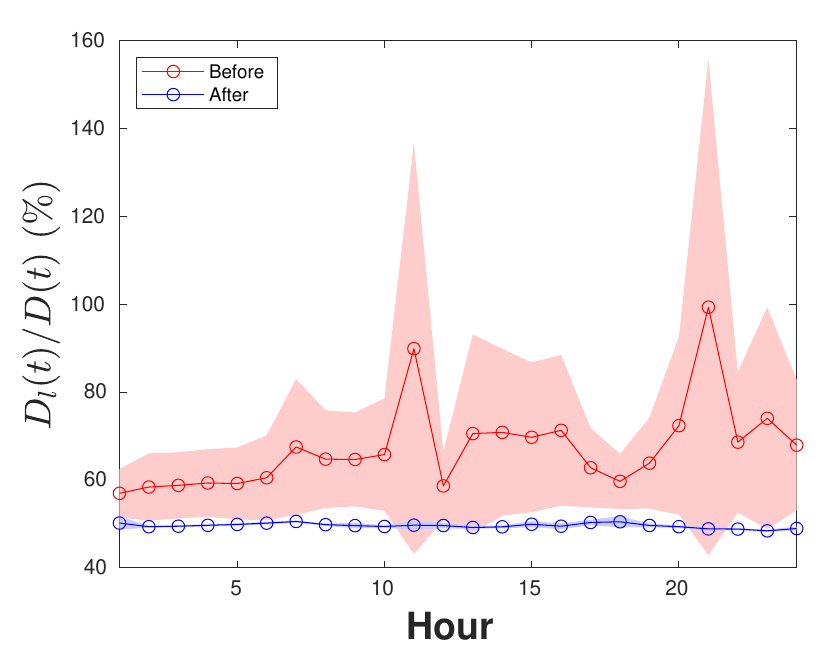}
    \caption{Demand Ratios for NYISO}
  \end{subfigure}
\caption{Average hourly ratios of DA demand to RT demand before and after the implementation of virtual bidding.}
\label{fig:AverageDemandHourly}
\end{figure}

Figure \ref{fig:AverageDemandHourly} reports the results for the demand
ratio. In CAISO, demand ratios before implementing virtual trading vary
between 45\% and 53\%, which suggests that nearly half of the demand would
be cleared in the DA market. However, this ratio declines by 10\% to 15\%
after the implementation of virtual trading. In NYISO, the demand ratios range from 59\% to 95\% before the introduction of virtual trading and from 53\% to 55\% afterward. This marked decline confirms our model’s prediction that virtual bidding significantly reduces demand ratios.\looseness = -1 


\section{Conclusion}
\label{sc:Conclusion}
In this study, we investigate the impact of virtual trading 
in two-settlement 
wholesale electricity markets. 
While this 
two-settlement 
structure provides hedging opportunities
for market participants, it also introduces potential
inefficiencies. We have built a game-theoretical framework which accounts for strategic interactions between renewable suppliers, LSEs, and 
virtual traders. We have analyzed how virtual trades, purely
financial electricity transactions without any physical deliveries, can mitigate price inefficiencies by enabling financial institutions to arbitrage between DA and RT markets. 

The supply function equilibrium of our framework reveals that, in the
absence of virtual trading, LSEs 
can
exploit their market power to underbid in the DA market, leading to lower DA prices compared to those
expected in the RT market. Virtual trading  narrows, and
potentially eliminate, the price spread between the two markets as the
presence of virtual traders grows. Nevertheless, this convergence in
prices does not necessarily imply a convergence in traded
quantities. LSEs may continue to bid below
their true demand estimates. Additionally, the inclusion of renewable
energy suppliers does not offset the strategic bidding behavior of LSEs.  
We have provided empirical support for our model predictions, using data from
 CAISO and NYISO.\looseness = -1 

\section*{Acknowledgments}{
This research was developed as a part of the grant “Risk-Aware Power System Control, Market and Market
Incentives”. The authors gratefully acknowledge the support of the Department of Energy through an ARPA-E
PERFORM program.}

\newpage
\bibliography{reference}
\bibliographystyle{informs2014}


\newpage
\appendix
\setcounter{page}{1}
\renewcommand{\thepage}{OA-\arabic{page}}
\newpage

\begin{center}
{\LARGE Online Appendix}
\end{center}
\section{Robustness with respect to Assumption
  \ref{Ass:RT}}\label{sec:robustmodelass} 
In the following discussion, we relax Assumption \ref{Ass:RT} to extend
our analyses to more general scenarios in Case 2. We first relax Assumption
\ref{Ass:RT}-(i) by 
allowing for a
residual demand range
$[\underline{D}_r,\overline{D}_r]$ where $\overline{D}_r>\underline{D}_r >
0$. In such case, the initial market clearing price $\gamma_0$,
at which $\sum_{i\in\mathcal{S}}\bar{S}^*_i(\gamma_0)+\bar{C}(\gamma_0)
-C(p_f^*) = \underline{D}_r$, may not 
be
zero. 
To deal with this issue, we first establish the following result.
\looseness
= -1 
\begin{restatable}[\textbf{Relaxation of Assumption 1-(i)}]{lemma}{RTInitialPrice}
\label{RT:InitialPrice_Exd}
Suppose 
the minimal residual demand 
$\underline{D}_r>0$.
Let $\{\bar{S}^\ast_i: i \in \bigS\}$ denote any SFE for this market, and
let $\gamma_0$ denote the market clearing price at the lowest residual
demand, i.e. $\sum_{i\in\mathcal{S}}\bar{S}^*_i(\gamma_0)+\bar{C}(\gamma_0) -C(p_f^*) =
\underline{D}_r$.
Then the following conclusions hold.
\begin{enumerate}[(i)]
\item All the renewable suppliers submit symmetric equilibrium supply
functions until their residual capacities are reached, i.e. 
the 
equilibrium supply 
function for the $i$-th renewable supplier is of the form
$\bar{S}_i^*(p)=\min\{\bar{S}(p),Q_{i}^r\}$, for $p \geq \gamma_0$, 
where $\bar{S}: [\gamma_0,\infty)\rightarrow\mathbb{R}^+$ is a continuous and
increasing function.
\item There is a unique SFE that maximizes the suppliers revenue.
\end{enumerate}
\end{restatable}
Lemma~\ref{RT:InitialPrice_Exd} 
extends the existence and uniqueness
results in Proposition \ref{RT:Existence} and \ref{RT:Uniqueness} by
relaxing Assumption 1-(i). Next, we relax Assumption 1-(ii) and consider a
scenario where the highest residual capacity $Q_{N_S}^r$ is never
reached in the RT market. Let $M_S \leq N_S$ denote the index such that the 
the residual capacity $Q_{M_S-1}^r$ is reached 
reached at $\overline{D}_r$ but 
$Q_{M_S}$ is not, i.e. $\sum_{i=1}^{M_S-1} Q_i^r \leq \overline{D}_r <
\sum_{i=1}^{M_S} Q_i^r $. 
In this case, the symmetry established in Lemma
\ref{RT:Symmetry} and \ref{RT:InitialPrice_Exd} still holds because we
used forward induction in establishing the results; however, the 
the forward induction
is
terminated at $M_S-1$ instead of $N_S-1$.  
Similarly, the solution \eqref{RT:Sol_0} of the ODE \eqref{RT:ODE} can be
derived 
for price interval $[\gamma_{i-1}(\bar{c}_{i-1}),
\gamma_i(\bar{c}_i)]$ for $i\in\{1,...,M_S-1\}$. Then we can ensure the
existence and uniqueness of $\{\sigma_i(p;\bar{c}_i)\}_{i=1}^{M_S-1}$ by
following the same methods in Proposition \ref{RT:Existence} and
\ref{RT:Uniqueness}. \looseness = -1 
  
Finally, we relax Assumption 1-(iii) to address scenarios where renewable
suppliers may not meet their DA commitments due to the capacity
constraints. Let $N_S^r  = \max\{i: Q^r_i = 0\} (\geq 1)$. Then 
suppliers $i \leq N_S^r$ 
are not strategic -- they fulfill their DA commitments
whenever feasible to avoid underproduction penalties; otherwise, they
would produce up to the residual capacity. The set of 
strategic renewable suppliers 
$\mathcal{S}_r=\mathcal{S}\setminus\{1,...,N_S^r\}$. 
The only difference 
here
compared to the analysis in
\S\ref{ssc:RTAnalysis} is that the number of strategic renewable
suppliers 
is
$N_S-N_S^r$ instead of $N_S$. By adjusting the initial index
from $k=1$ to $k=N_S^r+1$, we can apply the method used to derive 
\eqref{RT:Sol_0}-\eqref{RT:Sol_1} for suppliers indexed from $(N_S^r+1)$
to $N_S$.   
To determine the constants in the obtained supply functions
and ensure the uniqueness of the SFE, we replicate the SFE analysis
outlined in Proposition \ref{RT:Existence} and \ref{RT:Uniqueness}.  
 
\section{Solution of the ODE \eqref{RT:ODE}}
\label{sc:AppSolODE}
We first replace $p_s^*(t)$ in \eqref{RT:ODE} with a generic variable
$p\in\mathbb{R}^+$ to ease our notation, as there is no need to emphasize
the market clearing price $p_s^*(t)$ for time $t\in\mathcal{T}$ when
deriving the ODE solutions.  

Suppose we denote as $\sigma_1$ the solution to the ODE \eqref{RT:ODE}
when $k=1$, then \eqref{RT:ODE} becomes 
\begin{equation*}
\big(p(N_S-1)\big)\sigma'_1(p)-\sigma_1(p)=-p\bar{C}(p).
\end{equation*}
At $p=0$, the solution of the above ODE is $\sigma_1(p)=0$.
Define $\gamma_1:= \min\{p: \sigma_1(p)=Q_1^r\}$. For $p\in (0,\gamma_1]$,
dividing both sides by $p(N_S-1)$ we get 
\begin{equation*}
\sigma'_1(p)-\frac{1}{p(N_S-1)}\sigma_1(p)=\frac{1}{p(N_S-1)}\bigg(-p\bar{C}(p)\bigg).
\end{equation*}
This is a first order (nonhomogeneous) linear differential equation. Its solution is given by (see, e.g., page 47, \citealt{adkins2012ordinary}):
\begin{equation*}
\sigma_1(p)=
p^{\frac{1}{N_S-1}}\bigg(\frac{b_1}{N_S}-\frac{1}{N_S-1}\int_{a_1}^{p}\bar{C}'(x)x^{-1/(N_S-1)}dx\bigg),  
\end{equation*}
where $p\in (0,\gamma_1]$ and $b_1\in\mathbb{R}$ and $a_1\in(0,\gamma_1]$
are two constants. The integral exists due to the continuity of the
integrand. Define $F_1(x)$ as the anti-derivative of $\int
\bar{C}'(x)x^{-1/(N_S-1)} dx$, and $\bar{c}_1:=b_1/N_S+F_1(a_1)/(N_S-1)$. 
Then
\begin{align*}
\sigma_1(p;\bar{c}_1)=\begin{cases}
0, & \text{if}\ p=0; \\
p^{\frac{1}{N_S-1}}\Big(\bar{c}_1-\frac{1}{N_S-1}F_1(p)\Big), & \text{if}\
p\in(0,\gamma_1(\bar{c}_1)], 
\end{cases}
\end{align*}
where
$\gamma_1(\bar{c}_1):=\min\{p:\sigma_1(p;\bar{c}_1)=Q_1^r\}$,
and emphasize the dependence of $\sigma_1$ and $\gamma_1$on the constant $\bar{c}_1$. Using
L'Hopital's rule, we can check that the solution is right continuous at
$p=0$, i.e., 
\begin{equation}
\label{RT:RContinuous}
\lim_{p\rightarrow 0^+}p^{\frac{1}{N_S-1}}\Big(\bar{c}_1-\frac{1}{N_S-1}F_1(p)\Big) =0.
\end{equation}
The solutions for $k=2,\ldots,N_S-1$ can be derived in a similar way. 
The ODE~\eqref{RT:ODE} can be written as
\begin{equation*}
\sigma'_k(p)-\frac{1}{p(N_S-k)}\sigma_{k}(p)=\frac{1}{p(N_S-k)}\bigg(-p\bar{C}(p)\bigg).
\end{equation*}
The solution to the ODE is 
\begin{equation*}
\sigma_k(p;\bar{c}_k)= p^{\frac{1}{N_S-k}}\Big(\bar{c}_k-\frac{1}{N_S-k}F_k(p)\Big),
\end{equation*}
where $F_k$ is the anti-derivative of $\bar{C}'(x)x^{-1/(N_S-k)}$, 
where $p\in (\gamma_{k-1}(\bar{c}_{k-1}), \gamma_k(\bar{c}_k)]$, and
$\gamma_k(c_k):=\argmin\{p:\sigma_k(p;c_k)=Q_k^r\}$ for $1<k<N_S$.
We obtain \eqref{RT:Sol_0}--\eqref{RT:Sol_1} by combining $\sigma_1$ and
$\{\sigma_k\}_{k=2}^{N_S-1}$ as $\bar{S}$. 
\section{Proofs}
\label{sc:AppProof}
\RTSymmetry*
\begin{proof}
Let $\{\bar{S}^\ast_i(p): i = 1, \ldots, N_S\}$ denote an SFE for the RT
market. We establish the result by iterating over the number of renewable
suppliers at their capacity limits.
\begin{enumerate}[(i)]
\item Let $\gamma_1:=\min\{p: \exists i\in\mathcal{S}, \text{ s.t. }
  \bar{S}_i^*(p)=Q_1^r\}$. Then $\bar{S}_i^*(p)<Q_1^r$ for
  $p\in[0,\gamma_1)$ and $i\in\mathcal{S}$, i.e. no renewable supplier is
  at its residual capacity for $p \in [0,\gamma_1)$. Using the market
  clearing condition~\eqref{RT:MarketClearing}, we can rewrite the
  payoff~\eqref{RT:RS_Obj} of the $i$-th renewable supplier as  
  \begin{equation*}
    \big(D_r-\bar{C}(p) + C(p_f^*) -\bar{S}_{-i}^*(p)\big)p,
  \end{equation*}
  where
  $\bar{S}_{-i}^*(p):=\sum_{j\in\mathcal{S}\setminus\{i\}}\bar{S}^*_{j}(p)$.
  Note that $C(p_f^*)$ is a known constant as it is the conventional
  supplier's DA commitment. Given $D_r$ and $\bar{S}_{-i}^*$,
    the $i$-th renewable supplier maximizes its payoff by selecting an
    equilibrium supply function $\bar{S}_i^*$ that satisfies the FOC of
    the payoff:
    \begin{equation*}
      D_r- \bar{C}(p) + C(p_f^*) -\bar{S}_{-i}^*(p)+p\big(-\bar{C}'(p)-(\bar{S}_{-i}^*(p))'\big)=0,
    \end{equation*}
    which is equivalent to 
    \begin{equation}
      \label{App:L1_FOC}
      \bar{S}^*_i(p)+p\big(-\bar{C}'(p)-(\bar{S}_{-i}^*(p))'\big)=0
    \end{equation}
    because of \eqref{RT:MarketClearing}. 
Substituting 
\begin{equation*}
(\bar{S}_{-i}^*(p))'=\sum_{j\in\mathcal{S}}(\bar{S}_j^*(p))'-(\bar{S}_i^*(p))'
\end{equation*}
in \eqref{App:L1_FOC}, we obtain
\begin{equation*}
\bar{S}_i^*(p)+p\bigg(-\bar{C}'(p)-\sum_{i\in\mathcal{S}}(\bar{S}_i^*(p))'+(\bar{S}_i^*(p))'\bigg)=0,
\end{equation*}
which implies 
\begin{eqnarray}
  p\big(\bar{S}_i^*(p)\big)'+\bar{S}_i^*(p) & = & p\left(\bar{C}'(p) + \sum_{i\in\mathcal{S}}\big(\bar{S}_i^*(p)\big)'\right);  \label{App:ResidualDemand_Relation} \\
  \left(p\bar{S}_i^*(p)\right)' & = & p\bar{C}'(p) + p\sum_{i\in\mathcal{S}}\big(\bar{S}_i^*(p)\big)';\notag\\ 
p\bar{S}_i^*(p) & = & \int_{0}^{p}[x\bar{C}'(x)+x\sum_{i\in\mathcal{S}}\big(\bar{S}_i^*(x)\big)']dx+c_i,\label{App:L1_Symmetry} 
\end{eqnarray}
where $c_i\in\mathbb{R}$ is a heterogeneous constant for the $i$-th renewable supplier. 
Evaluating the above expression at $p=0$ implies that $c_i=0$ for $i\in\mathcal{S}$ in \eqref{App:L1_Symmetry} and that
the renewable suppliers have an identical equilibrium supply function, i.e., 
\begin{equation*}
\bar{S}_1^*(p) = \ldots  = \bar{S}_{N_S}^*(p) = \bar{S}(p), \qquad p \in [0,\gamma_1).
\end{equation*}
Note that, by continuity,  
$\bar{S}(\gamma_1) = Q_1^r$, and
\eqref{App:L1_FOC} implies that 
\[
p(N_S-1) \bar{S}'(p) - \bar{S}(p) = - p\bar{C}'(p), \qquad p \in [0,\gamma_1).
\]
Since $\bar{S}^*_1$ is assumed to be nondecreasing 
and continuous, it follows that $\bar{S}_1^*(p)=Q_1^r$ for
$p\geq\gamma_1$. Moreover, $\bar{S}^*_i(\gamma_1) = \bar{S}(\gamma_1) <
Q^r_i$ for $i \in\mathcal{S}\setminus\{1\}$.

\item  Let $\gamma_2:=\min\{p: \exists i \in\mathcal{S}\setminus\{1\}, \text{ s.t. } \bar{S}^*_i(p)=Q_i^r\}$. Since $\bar{S}^*_i$ is continuous and nondecreasing and $\bar{S}^*_i(\gamma_1) < Q^r_i$ for all $i \in\mathcal{S}\setminus\{1\}$, it follows that $\gamma_2 > \gamma_1$. By repeating the analysis in (i), we obtain
\begin{equation}
\label{App:L1_Symmetry_2}
\bar{S}_i^*(p)p=\int_{\gamma_1}^{p}
\Big[x\bar{C}'(x)+x\sum_{i=2}^{N_S}\big(\bar{S}^*_i(x)\big)' \Big]dx +c_i, \quad \forall p \in [\gamma_1, \gamma_2),
\end{equation}
for all $i \in\mathcal{S}\setminus\{1\}$. By letting $p=\gamma_1$, and using the fact that $\bar{S}_i^*(\gamma_1) = Q^r_1$ for $i \in\mathcal{S}\setminus\{1\}$, we obtain\looseness =-1
\begin{align*}
c_i = \bar{S}_i^*(\gamma_1)\gamma_1 = Q_1^r\gamma_1, \quad \forall i \in\mathcal{S}\setminus\{1\}.
\end{align*}
Thus, it follows that for $p \in [\gamma_1, \gamma_2)$, we have
\begin{equation*}
\bar{S}_2^* = \ldots = \bar{S}_{N_S}^*=\bar{S}(p), 
\end{equation*}
where $\bar{S}$ is a solution to
\[
p(N_S-2) \bar{S}'(p) - \bar{S}(p) = - p\bar{C}'(p), \quad p \in [\gamma_1, \gamma_2).
\]
By continuity, we have that $\bar{S}(\gamma_2) = Q_2^r$. Since $\bar{S}^*_2$ is assumed to be nondecreasing and continuous, it follows that $\bar{S}_2^*(p)=Q_2^r$ for $p\geq\gamma_2$. Moreover, $\bar{S}^*_i(\gamma_2) =
Q^r_2 < Q^r_i$ for $i \in\mathcal{S}\setminus\{1,2\}$.

\item More generally, for $k = 3, \ldots, N_S - 1$, let $\gamma_k:= \min\{p: \exists i\in\mathcal{S}\setminus\{1,...,k-1\}, \text{ s.t. }  S^*_i(p) = Q^r_i\}$. Since $\bar{S}^*_i$ is continuous and $\bar{S}^*_i(\gamma_{k-1}) < Q^r_i$ for all $i \in\mathcal{S}\setminus\{1,...,k-1\}$, it follows that $\gamma_k > \gamma_{k-1}$. Repeating the analysis in (i) and (ii), we get that for $p \in [\gamma_{k-1}, \gamma_k)$, 
\begin{equation*}
\bar{S}_2^* = \ldots = \bar{S}_{N_S}^*=\bar{S}(p), 
\end{equation*}
where $\bar{S}$ satisfies
\begin{equation*}
p(N_S-k) \bar{S}'(p) - \bar{S}(p) = - p\bar{C}'(p), \quad p \in [\gamma_{k-1},\gamma_k).
\end{equation*}
\end{enumerate}
Thus, we have established the symmetry result for the equilibrium supply function. Besides, $\bar{S}$ is increasing and continuous, with
$\bar{S}(\gamma_0) = 0$ at $\gamma_0:=0$, and $\bar{S}(\gamma_k) = Q^r_k$
for $1 \leq k \leq N_S-1$.
\looseness=-1
\end{proof}

\RTexistence*
\begin{proof}
  In the proof of Lemma~\ref{RT:Symmetry}, we implicitly show that if
    $\bar{S}(p)$ is a SFE, there exist constants~$\bar{\bm{c}}$ that
    satisfy $(i)$ and $(ii)$ above.
  
  Next, suppose the function $\bar{S}(p;\bar{\bm{c}})$ satisfies $(i)$ and
  $(ii)$.  We show that $\{\bar{S}^{*}_i(p) =
  \min\{\bar{S}(p;\bar{\bm{c}}), Q^r_i: i \in \bigS\}$ is an SFE, i.e.
  the $i$-th renewable supplier will not deviate from $\bar{S}^{*}_i(p)$ when suppliers $j \neq i$ bid
  $S^{*}_j$.   
  Consider the best response for the $i$-th supplier when all other
  renewable suppliers $j \neq i$ stick to $\bar{S}^*_j$.  
  We consider the best response for supplier~$i$ in a expanded set of supply
  functions, namely continuous functions that satisfy capacity constraints,
  i.e. $\bar{S}_i(p) \leq Q^r_i$.  
  Suppose the best response strategy for
  supplier~$i$ in this feasible set is $\Psi_i(p_s^*)$, where $p_s^*$ is
  the market clearing price for residual demand $D_r$. Note that since the
  supply functions of all suppliers (both renewable and conventional) are
  fixed, supplier~$i$ can control the market clearing price, and impact
  its payoff. The payoff to supplier~$i$ at the market clearing price
  $p_s^*$ is given by  
  \begin{equation*}
    g_i(p_s^*) = \big(D_r-\bar{S}^*_{-i}(p_s^*) - \bar{C}(p_s^*) + C(p_f^*) \big)p_s^*,   
  \end{equation*}
 Note that $\bar{S}_{-i}$
  is smooth except at points $\gamma_j$ for $j\in\mathcal{S}$. Then
  derivative of the above payoff function is  
  \begin{eqnarray*}
    g'(p_s^*) & = & \bigg(D_r-\bar{S}^*_{-i}(p_s^*) -
                    \bar{C}(p_s^*) + C(p_f^*) \bigg) -
                    p_s^*\bigg(\big(\bar{S}^*_{-i}(p_s^*)\big)^{\prime} +
                    \bar{C}'(p_s^*)\bigg),\\ 
              & = & \Psi(p_s^*) -
                    p_s^*\bigg(\big(\bar{S}^*_{-i}(p_s^*)\big)^{\prime} +
                    \bar{C}'(p_s^*)\bigg), 
  \end{eqnarray*}
  We slightly abuse
  our notation by using $(\bar{S}^*_{-i})^{'}$ to represent both the left hand
  derivative of $\bar{S}_{-i}$ at $\gamma_j$ and its full derivative at the other
  smooth points. 
  Then  the ODE~\eqref{RT:ODE} 
  implies that 
  \begin{align*}
    p_s^*\bigg(\big(\bar{S}^*_{-i}(p_s^*)\big)^{\prime} + \bar{C}'(p_s^*)\bigg)=\begin{cases}
      \bar{S}(p_s^*), & \text{if}\ \ p_s^* \in (\gamma_{0},\gamma_i],\\
      \bar{S}(p_s^*) + p_s^*\bar{S}'(p_s^*), & \text{if}\ \ p_s^* \in (\gamma_i,\gamma_{N_S}].
\end{cases}
\end{align*}
Thus, it follows that
\begin{align*}
g'(p_s^*)  = 
  \begin{cases}
    \Psi(p_s^*) - \bar{S}(p_s^*), & \text{if}\ \ p_s^* \in (\gamma_{0},\gamma_i],\\
    \Psi(p_s^*) - \bar{S}(p_s^*) - p_s^*\bar{S}'(p^*), & \text{if}\ \ p_s^* \in (\gamma_i,\gamma_{N_S}].
  \end{cases}
\end{align*}
Next, we consider the following two cases:
\begin{enumerate}[(a)]
\item $p_s^* \in (\gamma_{0},\gamma_i]$: Suppose $\Psi(p_s^*) = Q^r_i$,
  that is, the $i$-th supplier produces at its full capacity. In this
  case, the supplier can only reduce the supply locally, or equivalently
  increase the market clearing price (because the slack has to be picked
  up by the other renewable suppliers and the conventional
  supplier). 
  Since $\Psi$ is the optimal supply function for the supplier, it follow
  that 
  $0 \leq g_i'(p_s^*) = \Psi(p_s^*) - \bar{S}(p_s^*) = Q^r_i - \bar{S}(p_s^*)
> 0$, since $\bar{S}(p_s^*) < Q^r_i$ for $p^*\in
(\gamma_{0},\gamma_i]$. A contradiction. Therefore, $\Psi(p_s^*) < Q^r_i$. 
  
Since $\Psi(p_s^*) < Q^r_i$, it is possible for the supplier~$i$ to both
increase and decrease~$p_s^*$. Thus, it follows that $g'(p_s^*) =
\Psi(p_s^*) - \bar{S}(p_s^*) = 0$, i.e., $\Psi(p_s^*) = \bar{S}(p_s^*)$,
for $p^* \in (\gamma_{0},\gamma_i]$. 
 
\item $p_s^* \in (\gamma_i,\gamma_{N_S}]$. Suppose $\Psi(p_s^*) <
  Q^r_i$. Then following the argument above, we must have $g_i'(p_s^*) =
  0$, i.e., $\Psi_i(p_s^*) = \bar{S}(p_s^*)  + p_s^* \bar{S}'(p_s^*) >
  \bar{S}(p_s^*) \geq Q^r_i$, where the strict inequality is due to the
  fact that $\bar{S}'(p^*) > 0$. A contradiction. Thus, we must have that
  $\Psi(p_s^*) = Q^r_i$. 
\end{enumerate}
Since $\bar{S}$ is continuous and increasing, with $\bar{S}(\gamma_i) = Q^r_i$, it
follows that $\Psi(p_s^*) = \min\{\bar{S}(p_s^*),Q^r_i\}$.

Next, we show that the constants $(\bar{c}_2,\ldots,\bar{c}_{N_S-1})$ are
increasing functions of $\bar{c}_1$ when the continuity conditions 
\begin{equation}
\label{eq:IncreasingConst}
\sigma_{k}(\gamma_k(\bar{c}_k);\bar{c}_k) =
\sigma_{k+1}(\gamma_k(\bar{c}_k);\bar{c}_{k+1}),
\end{equation}
hold for $k = 1, \ldots, N_S-2$. 
For
$k=1$,  since
$\sigma_{1}(\gamma_1(\bar{c}_1);\bar{c}_1) = \sigma_{2}(\gamma_1(\bar{c}_1);\bar{c}_{2})$,
\eqref{RT:Sol_0}--\eqref{RT:Sol_1} imply that 
  \begin{equation*}
    \gamma_1(\bar{c}_1)^{\frac{1}{N_S-1}}\left(\bar{c}_1
      -\frac{1}{N_S-1}F_1\big(\gamma_1(\bar{c}_1)\big)\right)  
    = 
    \gamma_1(\bar{c}_1)^{\frac{1}{N_S-2}}
    \left(\bar{c}_2+\frac{1}{N_S-2}F_{2}\big(\gamma_1(\bar{c}_1)\big)\right). 
\end{equation*}
Rearranging, we obtain
\begin{equation}
\label{Prop1:c_1}
\bar{c}_2 = \gamma_1(\bar{c}_1)^{-\frac{1}{(N_S-1)(N_S-2)}} \left(\bar{c}_1-\frac{1}{N_S-1}F_1\big(\gamma_1(\bar{c}_1)\big)\right)
-\frac{1}{N_S-2}F_{2}\big(\gamma_1(\bar{c}_1)\big). 
\end{equation}
Thus, \eqref{eq:IncreasingConst} implies that  
$\bar{c}_2$ 
is a function of $\bar{c}_1$. Furthermore, $\bar{c}_2$ is an  
increasing function of $\bar{c}_1$ because 
$(\gamma_1(\bar{c}_1))^{-\frac{1}{(N_S-1)(N_S-2)}}$  is increasing in
$\bar{c}_1$, 
and 
$F_1(\gamma_1(\bar{c}_1))$ and
$F_{2}(\gamma_1(\bar{c}_1))$ are decreasing in $\bar{c}_1$
since $\gamma_1(\bar{c}_1)$ is decreasing in $\bar{c}_1$, as proved previously, and $F_i'(p) = p^{-1/(N_S-i)}\bar{C}'(p) > 0$ for all $p$ and $i=1, 2$. 
Repeating this 
analysis for $k=2,\ldots,N_S-1$, we establish a recursion that $\bar{c}_k$
is increasing in $\bar{c}_{k-1}$. Thus, 
$\bar{c}_k$, 
$k = 1, \ldots, N_S-1$, 
are increasing functions of $\bar{c}_1$.
  
Next, we show that the set of constants $\bar{\bm{c}}$ that satisfy $(i)$
and $(ii)$ is not an empty set, i.e., there always exists an SFE. 
Let $\sigma_k(p)$ denote the solution of the ODE
\begin{equation}
  \label{eq:SFE-ODE}
  p(N_S-k) \sigma_k'(p) - \sigma_k(p) = - p\bar{C}'(p), \quad k = 1, \ldots, N_S-1.
\end{equation}
In Appendix \ref{sc:AppSolODE}, we establish that 
\begin{equation*}
  \sigma_k(p;\bar{c}_k) =
  \begin{cases}
    0, & p = 0; \\
    p^{\frac{1}{N_S-k}}\bigg(\bar{c}_k-\frac{1}{N_S-k}F_{k}(p)\bigg), &
    p > 0,
  \end{cases}
\end{equation*}
where $\bar{c}_k$ is a constant and $F_{k}(p)$ is an antiderivative of
$\big(p\big)^{-1/(N_S-k)}\bar{C}'(p)$.  
The derivative
\begin{equation}
  \label{App:Derivative}
  \sigma_k'(p;\bar{c}_k)
    =\frac{1}{N_S-k}\left[p^{\frac{1}{N_S-k}-1}\left(\bar{c}_k-\frac{1}{N_S-k}F_k(p)\right)-\bar{C}'(p)\right]  
  \end{equation}
  is an increasing linear function of the constant $\bar{c}_k$. So,
  $\sigma_k'(p;\bar{c}_k) > 0$, if and only if, 
  \begin{equation*}
    \label{App:P1_Inequality}
    \bar{c}_k>\left(p^{\frac{N_S-(k+1)}{N_S-k}}\bar{C}'(p) +\frac{1}{N_S-k}F_k(p)\right).
  \end{equation*}
  Thus, $\sigma_k'(p;\bar{c}_k)>0$ for all $p \in [0,p_k]$ and $p_k > 0$, if and only if 
  \begin{align}
    \label{App:P1_c1}
    \bar{c}_k> \underline{c}_k(p_k):=\max_{p\in[0,p_k]}\left\{p^{\frac{N_S-(k+1)}{N_S-k}}\bar{C}'(p)+\frac{1}{N_S-k}F_k(p)\right\},
  \end{align} 
  where $\underline{c}_k(p_k) < \infty$ because $F_k$ is bounded as the
  limit $\lim_{p\rightarrow 0^+}p^{1/(N_S-k)-1}F_k(p)$ exists and
  $\bar{C}'$ is bounded by assumption. Note that, by definition,
  $\underline{c}_k(p_k)$ is a non-decreasing function of $p_k$. 

According to \eqref{App:P1_c1}, $\sigma_k(p;\underline{c}_k(p_k))$
increases on $[0,p_k]$ and achieves its maximum value at $p=p_k$. Choose
$p_k$ such that $\sigma_k(p_k;\underline{c}_k(p_k))=Q_k^r$. Then
the monotonicity of $\sigma_k(p;\underline{c}_k)$ holds for any $\bar{c}_k
> \underline{c}_k(p_k)$ because we have $\sigma_k'(p;\bar{c}_k) >
\sigma_k'(p,\underline{c}_k(p_k)) > 0$ for all $p \in [0,p_k]$. In other
words, increasing $\bar{c}_k$ does not change the monotonicity of
$\sigma_k(p;\bar{c}_k)$ as long as $\bar{c}_k > \underline{c}_k(p_k)$. In
addition, $\sigma_k(p;\bar{c}_k) > \sigma_k(p;\underline{c}_k(p_k))$ for
$\bar{c}_k > \underline{c}_k(p_k)$. That is, $\sigma_k(p;\bar{c}_k)$
increases with $\bar{c}_k$.

Let
\begin{equation}
  \label{eq:gammak-def}
  \gamma_k(c):=\min\{p:\sigma_k(p;c)=Q_k^r\}
\end{equation}
denote the function that defines the price $p$ at which
$\sigma_k(\cdot;c)$ reaches $Q_k^r$ as a function of $c$. Then, it is clear
that $\gamma_k(\bar{c}_k) < p_k$ for $\bar{c}_k
> \underline{c}_k(p_k)$,
and hence, $\sigma_k'(p;\bar{c}_k) > 0$ for all $p \in
[0,\gamma_k(\bar{c}_k)]$ and $k\in\{1,...,N_S-1\}$. In addition, by definition, $\gamma_k(\bar{c}_k)$
is continuous and decreases with $\bar{c}_k$ due to the fact that
$\sigma_k(p;\bar{c}_k)$ becomes ``steeper" with larger $\bar{c}_k$. In the limit, as
$\bar{c}_k\rightarrow\infty$, we have
$\sigma_k'(p,\bar{c}_k)\rightarrow\infty$ according to
\eqref{App:Derivative} and $\gamma_k(\bar{c}_k)\rightarrow 0$ by its
definition. 

Next, we use a forward and backward induction approach to ``stitch" the
supply functions $\{\sigma_k(p;\bar{c}_k)\}_{k=1}^{N_S-1}$ together. In
the forward induction, given a set of constants 
$(\bar{c}_1,...,\bar{c}_{k-1})$ determined by the analysis above, we
choose the constant $\bar{c}_{k}$ such 
that the supply function $\sigma_{k}(p;\bar{c}_{k})$ reaches $Q_{k-1}^r$ 
at a price lower than that of the function  $\sigma_{k-1}(p;\bar{c}_{k-1})$, i.e., 
\begin{equation}
\label{app:Prop1_p}
\chi_{k-1}(\bar{c}_{k}):=\min\{p:\sigma_{k}(p;\bar{c}_{k})=Q_{k-1}^r\} \leq \gamma_{k-1}(\bar{c}_{k-1}).
\end{equation}
for all $k\in\mathcal{S}\setminus\{1,N_S\}$. To this end, for $k=1$, we
choose
$\bar{c}_1 >\bar{c}_1(p_1)$, where $p_1$ makes $\sigma_1(p_1;\bar{c}_1(p)_1)=Q_1^r$.  For $2\leq k\leq N_S-1$,
choose
\begin{equation}
\label{app:Prop1_c_2}
\bar{c}_k \geq \min\bigg\{c: c\geq \bar{c}_k(p_k)\ \text{and} \ \min\{p:\sigma_{k}(p;c)=Q_{k-1}^r\} \leq \gamma_{k-1}(\bar{c}_{k-1})\bigg\},
\end{equation}
where
the first condition 
ensures the monotonicity of the supply function 
and
the second condition is the requirement \eqref{app:Prop1_p}. The constant
$\bar{c}_k$ is iteratively determined because the recursion relies on
$\bar{c}_{k-1}$ and $\gamma_{k-1}(\bar{c}_{k-1})$ determined in the
previous iteration. The right hand side is always nonempty, because, as
$c\rightarrow\infty$, the first condition is certainly met; the second  condition is also satisfied due to
\begin{equation*}
\lim_{c\rightarrow\infty} \min\{p:\sigma_{k}(p;c)=Q_{k-1}^r\} = 0.
\end{equation*}
In the backward induction, we start from $k=N_S-1$ and fix
$\bar{c}_{N_S-1}$ as discussed above. Then from $k=N_S-2$ to $k=1$, we
iteratively set the constant $\bar{c}_k$ as  
\begin{equation}
\label{app:Prop1_c_3}
\bar{c}_k = \min\{c: \sigma_{k}(\chi_k(\bar{c}_{k+1});c)=Q_k^r\}, 
\end{equation}
i.e., we 
require that the supply function $\sigma_{k}(p;\bar{c}_k)$ reaches
$Q_k^r$ at $p=\chi_k(\bar{c}_{k+1})$. 
It is important to note that since $\chi_k(\bar{c}_{k+1})\leq
\gamma_k(\bar{c}_k)$, we need to further increase $\bar{c}_k$ determined
by \eqref{app:Prop1_c_2} to satisfy \eqref{app:Prop1_c_3}. Thus, the
monotonicity of $\sigma_k(p;\bar{c}_k)$ is preserved. Since the functions
$\{\sigma_k(\cdot;\bar{c}_k)\}_{k=1}^{N_S-1}$ are continuous and
increasing, the prices $\{\gamma_k(\bar{c}_k)\}_{k=1}^{N_S-1}$ follows an
increasing order: 
\begin{equation*}
0:=\gamma_0(\bar{c}_0) <\gamma_1(\bar{c}_1)<...<\gamma_{N_S-1}(\bar{c}_{N_S-1}).
\end{equation*}
The ``stitched'' function is defined as
\begin{equation}
  \label{eq:stitched-fn}
  \bar{S}(p) := \sigma_k(p;\bar{c}_k), \qquad p \in
  [\gamma_{k-1}(\bar{c}_{k-1}),\gamma_{k}(\bar{c}_{k})],\ k = 1, \ldots, N_S-1.
\end{equation}
Using L'Hopital's rule, we can check that
$\lim_{p\rightarrow 0^+}\bar{S}(p) = 0$, and 
$\lim_{p\rightarrow 0^+}\bar{S}'(p) = 0$. Thus, we have established that
the set of SFEs is not empty. 
\end{proof}

\RTUniqueness*
\begin{proof}
Since the derivative
\begin{equation*}
\sigma_k'(p;\bar{c}_k)=\frac{1}{N_S-1}\Big[p^{\frac{1}{N_S-k}-1}\bigg(\bar{c}_1-\frac{1}{N_S-k}F_k(p)\bigg)-\bar{C}'(p)\Big],
\end{equation*}
is a linear function of $\bar{c}_1$ and $(p)^{1/(N_S-k)-1}>0$, it follows
that the supply function $\sigma_1(p;\bar{c}_k)$ becomes ``steeper" as
$\bar{c}_k$ increases. This, in conjunction with
Proposition~\ref{RT:Existence}, implies that $\sigma_k'(p;\bar{c}_k)$ is
an increasing function of $\bar{c}_1$. That is, if we increase $\bar{c}_1$,
the entire supply function
$\bar{S}(p)$ 
becomes ``steeper". 

Let $p^\ast_s(\bar{c}_1)$ denote the market clearing price as a function of $\bar{c}_1$. For a given demand $D(t)$, the market clearing condition \eqref{RT:MarketClearing} implies that a steeper $\bar{S}$ results in a smaller market clearing price, and therefore, $p^\ast_s(\bar{c}_1)$ decreases in $\bar{c}_1$. 

Next, consider the payoff $\Pi^{\mbox{RT}}_i(\bar{c}_1)$ of the $i$-th renewable supplier $i$ as a function of $\bar{c}_1$
\begin{equation*}
\Pi^{\mbox{RT}}_i(\bar{c}_1) = \Big(D_r- \bar{C}(p_s^*(\bar{c}_1)) + C(p_f^*) 
- \bar{S}^\ast_{-i}(p_s^*(\bar{c}_1);\bar{c}_k(\bar{c}_1))\Big)
p_s^{*}(\bar{c}_1),
\end{equation*}
where we explicitly account for the fact that the equilibrium supply function $\bar{S}^\ast_j(\cdot;\bar{c}_k(\bar{c}_1))$ for $j\in\mathcal{S}\setminus\{i\}$ depends on $\bar{c}_1$.
Then the derivative w.r.t. $\bar{c}_1$ is
\begin{align}
&\frac{\mathrm{d}\Pi^{\text{RT}}_i(\bar{c}_1)}{\mathrm{d}\bar{c}_1}
  = \ \frac{\mathrm{d}p_s^{*}(\bar{c}_1)}{\mathrm{d}\bar{c}_1}\Big(D_r-
      \bar{C}(p_s^*(\bar{c}_1))+
      C(p_f^*)
      -\bar{S}^\ast_{-i}(p_s^*(\bar{c}_1);\bar{c}_k(\bar{c}_1))\Big)\notag\\ 
    &\hspace{2cm}- p_s^*(\bar{c}_1)\Big(\bar{C}'(p_s^*(\bar{c}_1)) +
      \Big(\bar{S}^\ast_{-i}(p_s^*(\bar{c}_1);\bar{c}_k(\bar{c}_1))\Big)'\Big)\frac{\mathrm{d}
      p_s^{*}(\bar{c}_1)}{\mathrm{d}\bar{c}_1} \notag\\
      &\hspace{2cm}- p_s^*(\bar{c}_1)
      \frac{\partial
      \bar{S}^*_{-i}(p_s^*(\bar{c}_1);\bar{c}_k(\bar{c}_1)))}{\partial
      \bar{c}_k}\frac{\mathrm{d}
      \bar{c}_k(\bar{c}_1)}{\mathrm{d}\bar{c}_1}\notag\\ 
=& \frac{\mathrm{d} p_s^{*}(\bar{c}_1)}{\mathrm{d}\bar{c}_1} \underbrace{\Big(D_r- \bar{C}(p_s^*(\bar{c}_1)) + C(p_f^*) - \bar{S}^\ast_{-i}(p_s^*(\bar{c}_1);\bar{c}_k(\bar{c}_1)) - p_s^*(\bar{c}_1)\Big(\bar{C}'(p_s^*(\bar{c}_1)) + \Big(\bar{S}^\ast_{-i}(p_s^*(\bar{c}_1);\bar{c}_k(\bar{c}_1))\Big)'\Big)\Big)}_{(a)} \notag\\
&- p_s^*(\bar{c}_1)\sum_{j\neq i}\underbrace{\frac{\partial \bar{S}^*_{j}(p_s^*(\bar{c}_1);\bar{c}_k(\bar{c}_1))}{\partial \bar{c}_k}\frac{\mathrm{d} \bar{c}_k(\bar{c}_1)}{\mathrm{d}\bar{c}_1}}_{(b)}. \label{RT:Derivative_c_1}
\end{align}
In the above formula, $\frac{\mathrm{d}p_s^{*}(\bar{c}_1)}{\mathrm{d}\bar{c}_1}<0$ becasue $p_s^{*}(\bar{c}_1)$ is decreasing in $\bar{c}_1$ as stated earlier. The parenthetical term $(a)$ in
the first line of \eqref{RT:Derivative_c_1} is always non-negative because
it is the FOC of the $i$-th supplier. Note that the FOC is zero before the
$i$-th supplier reaches its capacity and nonnegative when the supplier is
at the capacity (otherwise the supplier would bid less than $Q_i^r$ to
increase its payoff, which contradicts the result in Proposition
\ref{RT:Existence}). Moreover, the term $(b)$ is positive because
$\frac{\mathrm{d}\bar{c}_k(\bar{c}_1)}{\bar{c}_1} >0$ according to Proposition
\ref{RT:Existence}, and $\frac{\partial
  \bar{S}^*_{j}(p_s^*(\bar{c}_1);\bar{c}_k(\bar{c}_1))}{\partial
  \bar{c}_k}= \frac{1}{N_S-k}(p_s^*(\bar{c}_k))^{1/(N_S-k)-1}>0$, for
$p_s^*(\bar{c}_k)\in [\gamma_{k-1},\gamma_k]$ and $k=1,...,N_S-1$, due to
\eqref{RT:Sol_1}. Hence,
$\frac{\mathrm{d}\Pi^{\text{RT}}_i(\bar{c}_1)}{\mathrm{d}\bar{c}_1}<0$,
i.e., the supplier's payoff is decreasing in $\bar{c}_1$ and choosing the
smallest $\bar{c}_1$ that guarantees the increasing and continuity of the
supply function maximizes the payoff. \looseness = -1
\end{proof}

\RTInitialPrice*

\begin{proof}

  All the suppliers bid supply functions $S_{i}(p)$ that are continuous,
  piecewise differentiable, 
  and non-decreasing for $p \in [0,\infty)$. However, unlike in
  Lemma~\ref{RT:Symmetry} the FOC conditions are no longer valid
  for an interval of the form $[0,p_{\max}]$ for some $p_{\max} > 0$,
  since the smallest residual demand $\underline{D}_r > 0$. We show below
  that this does not impact that symmetry result. 

  Let $\bar{S}^{*}_i$, $ i \in \bigS$ denote any SFE. Then the FOC for
  optimality implies that
  \[
    \underline{S}_i^*(p)  - p \big( \sum_{j \neq i} (\bar{S}^*_j)'(p) +
    \bar{C}'(p)\big) 
    \begin{cases}
      \leq 0, & \bar{S}^*_i(p) = Q^r_i,\\
      = 0, & \bar{S}^*_i(p) < Q^r_i,
    \end{cases}
  \]
  where we abuse notation by using $(\bar{S}^*_i)'$ to denote the
  derivative, when it exists, or the correct directional derivative. 

  Let $\bigU = \{i: \bar{S}_i^*(\gamma_0) < Q^r_i\}$ denote the set of
  suppliers have residual capacity at the price $\gamma_0$. By continuity,
  it follows that $\gamma_1:=\min\{p: \bar{S}_i^*(p)=Q_i^r, \text{for some
  } i \in \bigU\} > \gamma_0$. Then, the argument used in
  Lemma~\ref{RT:InitialPrice_Exd} implies that for $p
  \in [\gamma_0,\gamma_1]$, the FOC can be rewritten as
  \[
    \big( p \bar{S}^*_i(p) \big)' = p \Big( \sum_{j \in \bigU}
    \bar{S}^*_j(p) + \bar{C}'(p)\Big).
  \]
  Thus, $p \bar{S}^*_i(p) = \gamma_0 S_i^*(\gamma_0) + \int_{\gamma_0}^p
  x\big( \sum_{j\in \bigU} (S_j^*)'(x) + \bar{C}'(x)\big) dx$. Since
  $\gamma_0 S_i^*(\gamma_0) $ can be chosen without impacting the FOC, 
  it follows that $S_i^*(\gamma_0)$ must an equilibrium at
  the price $\gamma_0$, or equivalently, 
  \[
    S_i^*(\gamma_0) \in {\argmax_q} \Big\{ q \bar{C}^{-1}\Big(\underline{D}_r -
    \sum_{j \not \in \bigU} Q^r_j - \sum_{j \in \bigU, j \neq i} \bar{S}^*_{{j}}(\gamma_0) - q\Big)\Big\},
  \]
  where we have used the market clearing condition~\eqref{RT:MarketClearing}. 
  Setting the first order derivative equal to zero, it follows that the maximum is attained at
  \[
    S_i^*(\gamma_0) = \frac{\bar{C}^{-1}\Big(\underline{D}_r
    - \sum_{j \not \in \bigU} Q^r_j - \sum_{j \in \bigU, {j \neq i}}
    \bar{S}^*_{{j}}(\gamma_0)\Big)}{\big(\bar{C}^{-1}\big)'\Big(\underline{D}_r
    - 
    \sum_{j \not \in \bigU} Q^r_j- \sum_{j \in \bigU, {j \neq i}} \bar{S}^*_{{j}}(\gamma_0)\Big)}.
  \]
  Thus, it follows that $\bar{S}^*_i$ is identical for all $i \in
  \bigU$. Now, following the argument in Lemma~\ref{RT:Symmetry} it
  follows that for for $p \geq \gamma_0$, we have $\bar{S}_i^*(p) = \min\{\bar{S}(p),Q^r_i\}$ for some
  non-decreasing, piecewise differentiable function $\bar{S}$.

  Recall that in Proposition~\ref{RT:Uniqueness}  we establish that the SFE corresponds to the lowest possible supply function. The same is true in the setting here. Thus, we would like $\gamma_0$, the price
  corresponding to $\underline{D}_r$ to be the largest possible, i.e.,
  $
  \gamma_0 = \bar{C}^{-1}(\underline{D}_r).
  $
  Note that $\gamma_0$, the supply from all the renewable suppliers is
  zero. Let $\bar{S}^{*}_0(p)$ denote the optimal supply function for when
  the residual demand $\bar{D}_r = 0$, i.e., the SFE identified in 
  the proof of Proposition~\ref{RT:Uniqueness}. Then, the SFE for
  the setting considered here is given by $\bar{S}^*_i(p)  =
  \bar{S}_{0}^*(\max\{p-\gamma_0,0\})$, $i \in \bigS$. 
\looseness=-1
\end{proof}
We need the following result to establish Proposition~\ref{DA:Existence}.  
\begin{lemma}
  \label{App:p_s_derivative}
  Suppose the RT market is active. Then the ex post RT equilibrium price
  $p_s^*$ is a decreasing function of the  DA price
  $p_f$, and $-\infty < \frac{\partial p_s^*}{\partial p_f} < 0$. 
\end{lemma}
\begin{proof}
 Let $D$ denote the realized demand. Then the ex post RT market clearing
 price $p_s^*$ satisfies the market clearing condition:
\begin{equation*}
\sum_{i=1}^{N_S}\min\{S_i(p_f),Q_i\}+\sum_{i\in\mathcal{S}_r}\bar{S}_i^*(p_s^*)+\bar{C}(p_s^*)
+ C(p_f) = D.
\end{equation*}
For any given $D$, $\{Q_i\}_{i=1}^{N_S}$, and $S_i$ for
$i\in\{1,...,N_S\}$, an increase in $p_f$ in the DA market leads to an
increase in DA supply. Since total demand is fixed, the RT supply must
decrease. Consequently, $p_s^*$ must decrease, implying
$\frac{\partial p_s^*}{\partial p_f} < 0$. Moreover, note that as the increase in the DA supply is
finite, $\frac{\partial p_s^*}{\partial p_f} \rightarrow -\infty$ can occur only if the aggregated RT
supply $\sum_{i\in\mathcal{S}}\bar{S}'_i(p_s^*)+\bar{C}'(p_s^*)\rightarrow
0$. However, this is impossible as $\bar{S}'_i(p_s^*)>0$ when $p_s^*>0$. 
\end{proof}

\DAExistence*
\begin{proof}
  The proof consists of two parts. We show the existence of the SFE in part
  (i), and then prove its uniqueness in part (ii).
\begin{enumerate}[(i)]
\item Let $h(x)$ the density function of the renewable capacity. Then the
  FOC (w.r.t. $p_f$) of the renewable supplier's payoff \eqref{DA:RS_Obj}
  is given by 
  \begin{eqnarray*}
  \lefteqn{(-S_{-i}'(p_f)-C'(p_f))p_f + S_i(p_f) +
    \mathbb{E}[p_s^*(p_s^*)'(\bar{S}_i^*(p_s^*))'] +
    \mathbb{E}[(p_s^*)'\bar{S}_i^*(p_s^*)]} \\ 
  && \mbox{} + (S_{-i}'(p_f)+C'(p_f))\int_{0}^{S_i(p_f)}p_s^*h(x)dx +
     \int_{0}^{S_i(p_f)}(x-S_i(p_f))(p_s^*)'h(x)dx = 0. 
  \end{eqnarray*}
  Using the DA market clearing condition \eqref{DA:MarketClearing} and the
  symmetry property \eqref{DA:Symmetry}, we can re-write the above FOC as
  \begin{eqnarray}
S_i'(p_f) & = & \frac{-\int_{0}^{S_i(p_f)}(x-S_i(p_f))(p_s^*)'h(x)dx-
                S_i(p_f)-\mathbb{E}[p_s^*(p_s^*)'(\bar{S}_i^*(p_s^*))'] -
                \mathbb{E}[(p_s^*)'\bar{S}_i^*(p_s^*)]}{(N_S-1)(-p_f +
                \int_{0}^{S_i(p_f)}p_s^*h(x)dx)} \notag\\
    && \mbox{}  - \frac{C'(p_f)}{(N_S-1)} \label{App:Prop3_ODE}
  \end{eqnarray}
  Meanwhile, using \eqref{DA:MarketClearing}, the LSE's payoff \eqref{DA:LSE_Obj} can be expressed as
\begin{equation*}
\bigg[\sum_{i\in\mathcal{S}}S_i(p_f)+C(p_f)\bigg]p_f+\mathbb{E}\bigg[\bigg(D- (S_i(p_f)+C(p_f))\bigg)^+p_s^*\bigg].
\end{equation*}
The corresponding FOC (w.r.t. $p_f$) is
\begin{align*}
&\bigg[\sum_{i\in\mathcal{S}}S_i(p_f^*)+C(p_f^*)\bigg]+\bigg[\sum_{i\in\mathcal{S}}S'_i(p_f^*)+C'(p_f^*)\bigg]p_f^* + \mathbb{E}\bigg[\bigg(-\sum_{i\in\mathcal{S}}S'_i(p_f^*)-C'(p_f^*)\bigg)p_s^*\bigg]\notag\\
&+\mathbb{E}\bigg[\bigg(\sum_{i\in\mathcal{S}}\bar{S}_i(p_s^*) +\sum_{i\in\mathcal{S}}(Q_i-S_i(p_f^*))^-+\bar{C}(p_s^*)-C(p_f^*)\bigg)^+(p_s^*)'\bigg] = 0,
\end{align*}
where the last term is obtained according to \eqref{RT:MarketClearing}.
We can rearrange the above FOC and obtain
\begin{align}
  & \sum_{i\in\mathcal{S}}S_i(p_f^*)+(p_f^*-\mathbb{E}[p_s^*])\sum_{i\in\mathcal{S}}S'_i(p_f^*)
  \notag \\
  &  \mbox{} + 
  \mathbb{E}\bigg[(p_s^*)'\sum_{i\in\mathcal{S}}\big(\bar{S}_i(p_s^*)
    +(Q_i-S_i(p_f^*))^-\big)\bigg|D>D_l\bigg]\mathbb{P}\bigg(D>D_l\bigg)\notag\\  
  & \mbox{} +
    C(p_f^*)+(p_f^*-\mathbb{E}[p_s^*])C'(p_f^*) +
    \mathbb{E}\bigg[(p_s^*)'(\bar{C}(p_s^*)-C(p_f^*))\bigg|D>D_l\bigg] \mathbb{P}\bigg(D>D_l\bigg)  
    = 0. \label{App:Prop3_LSE_FOC}  
\end{align} 
Using renewable suppliers' FOCs to simplify \eqref{App:Prop3_LSE_FOC}, we obtain
\begin{align}
  &\bigg(p_f^* - \mathbb{E}[p_s^*]\bigg)\sum_{i\in\mathcal{S}}S_i'(p_f^*)
    \notag \\
  & \mbox{} +
    \bigg(p_f^*-\mathbb{E}[p_s^*|Q_i<S_i(p_f^*)]\mathbb{P}(Q_i<S_i(p_f^*))\bigg)
    \bigg((N_S-1)\sum_{i\in\mathcal{S}}S_i'(p_f^*)+N_SC'(p_f^*)\bigg)\notag\\ 
  &\mbox{} +C(p_f^*) + \bigg(p_f^* -
    \mathbb{E}[p_s^*]\bigg)C'(p_f^*)+\mathbb{E}\bigg[(p_s^*)'(\bar{C}(p_s^*)
    - C(p_f^*))\bigg|D>D_l\bigg]\mathbb{P}\bigg(D>D_l\bigg)
    \notag\\ 
  & \mbox{} -
    \mathbb{E}[(p_s^*)'\sum_{i\in\mathcal{S}}\bar{S}'_i(p_s^*)]=0. \label{App:Prop3_ODE_2} 
\end{align}
Then an equilibrium supply function $S_i$ in the DA market should satisfy
both \eqref{App:Prop3_ODE} and \eqref{App:Prop3_ODE_2}. We prove the
existence of such supply function in the following two steps: \\
\noindent
\textbf{Step 1}: 
The proof in this step follows using a similar strategy as in Lemma EC. 1 in \cite{sunar2019strategic}. Suppose
  $(p_1,s_1)\in(0,\infty)\times (0,\infty)$ and
  $(p_2,s_2)\in(0,\infty)\times (0,\infty)$ are two points that satisfy
  \eqref{App:Prop3_ODE} and \eqref{App:Prop3_ODE_2}. We first show that
  there exists a unique trajectory that passes through each point for
  $p\in[p_i-\delta_i, p_i+\delta_i]$, where $\delta_i>0$ is a constant and
  $i\in\{1,2\}$. 

  Based on \eqref{App:Prop3_ODE}, we can define a mapping $f:
  (0,\infty)\times (0,\infty) \rightarrow\mathcal{R}^+$ as 
  \eq
  f(p,u) = \frac{1}{N_S-1}\bigg\{\frac{-\int_{0}^{u}(x-u)(p_s^*)'h(x)dx-
    u-\mathbb{E}[p_s^*(p_s^*)'(\bar{S}_i^*(p_s^*))'] -
    \mathbb{E}[(p_s^*)'\bar{S}_i^*(p_s^*)]}{-p +
    \int_{0}^{u}p_s^*h(x)dx} -C'(p)\bigg\}. 
  \en
  Then we have
  \begin{align*}
\frac{\partial f(p,u)}{\partial u} =\frac{1}{(N_S-1)(-p+\int_0^up_s^*h(x)dx)^2}\bigg[\bigg(\int_0^u(p_s^*)'h(x)dx-1\bigg)\bigg(-p+\int_0^up_s^*h(x)dx\bigg)\\
-(p^*_sh(u))\bigg(-\int_{0}^{u}(x-u)(p_s^*)'h(x)dx- u-\mathbb{E}[p_s^*(p_s^*)'(\bar{S}_i^*(p_s^*))'] - \mathbb{E}[(p_s^*)'\bar{S}_i^*(p_s^*)]\bigg)\bigg].
\end{align*}
Since $\partial f (p,u)/\partial u < \infty$ for $p \neq \int_{0}^u p_s^*h(x)dx$, $f$ is Lipschitz continuous in $u$ on $[u-\hat{\delta}, u+\hat{\delta}]$ for some constant $\hat{\delta}>0$. Then according to the ODE theory, there exists a unique trajectory that satisfies $S_i(p_1) = s_1$ and \eqref{App:Prop3_ODE} for $p\in[p_1-\delta_1, p_1+\delta_1]$, where $\delta_1>0$ is a constant. If $p = \int_{0}^u p_s^*h(x)dx$, then following a similar argument of Lemma EC. 1 in \cite{sunar2019strategic}, we can show that $1/f(p,u)$ is Lipschitz continuous in $p$. Therefore, there also exists a unique trajectory that passes through $(p_1, s_1)$. 

Similarly, based on \eqref{App:Prop3_ODE_2}, we have
\begin{align*}
S'_i(p_f) &= \frac{1}{N_S(p_f - \mathbb{E}[p_s^*])+(N_S-1)(p_f - \mathbb{E}[p_s^*|Q_i<S_i(p_f)]\mathbb{P}(Q_i<S_i(p_f)))}\bigg\{N_S\mathbb{E}[p_s^*(p_s^*)'(\bar{S}_i^*(p_s^*))']\\
&-N_S\bigg(p_f-\mathbb{E}[p_s^*|Q_i<S_i(p_f)]\mathbb{P}(Q_i<S_i(p_f))\bigg)C'(p_f)-C(p_f)-(p_f-\mathbb{E}[p_s^*])C'(p_f)\\
&-\mathbb{E}\big[(p_s^*)'\big(\bar{C}(p_s^*) -C(p_f)\big)\big]\bigg\}. 
\end{align*}
We define a mapping $g: (0,\infty)\times (0,\infty) \rightarrow\mathcal{R}^+$ as
\begin{align*}
g(p,u) &= \frac{1}{N_S(p - \mathbb{E}[p_s^*])+(N_S-1)(p - \mathbb{E}[p_s^*|Q_i<u]\mathbb{P}(Q_i<u))}\bigg\{N_S\mathbb{E}[p_s^*(p_s^*)'(\bar{S}_i^*(p_s^*))']\\
&-N_S\bigg(p-\mathbb{E}[p_s^*|Q_i<u]\mathbb{P}(Q_i<u)\bigg)C'(p)-C(p)-(p-\mathbb{E}[p_s^*])C'(p)\\
&-\mathbb{E}\big[(p_s^*)'\big(\bar{C}(p_s^*) -C(p)\big)\big]\bigg\}.
\end{align*}
The partial derivative $\partial g (p,u)/\partial u$ is
\begin{align*}
\frac{\partial g (p,u)}{\partial u} &= \frac{1}{\bigg(N_S(p - \mathbb{E}[p_s^*])+(N_S-1)(p - \mathbb{E}[p_s^*|Q_i<u]\mathbb{P}(Q_i<u))\bigg)^2}\bigg\{\\
&\bigg(N_S(p_s^*)h(s)C'(p)\bigg)\bigg(N_S(p - \mathbb{E}[p_s^*])+(N_S-1)(p - \mathbb{E}[p_s^*|Q_i<u]\mathbb{P}(Q_i<u))\bigg)\\ 
&+ \bigg(N_S\mathbb{E}[p_s^*(p_s^*)'(\bar{S}_i^*(p_s^*))'] -N_S(p-\mathbb{E}[p_s^*|Q_i<u]\mathbb{P}(Q_i<u))C'(p)-C(p)-(p-\mathbb{E}[p_s^*])C'(p)\\
&-\mathbb{E}\big[(p_s^*)'\big(\bar{C}(p_s^*) -C(p)\big)\big]\bigg)\bigg((N_S-1)p_s^*h(s)\bigg)
\bigg\}.
\end{align*}
Since $\partial g (p,u)/\partial u < \infty$, $g(p,u)$ is also Lipschitz continuous in $u$, if $N_S(p - \mathbb{E}[p_s^*])+(N_S-1)(p - \mathbb{E}[p_s^*|Q_i<u]\mathbb{P}(Q_i<u)) \neq 0$; otherwise, $1/g(p,s)$ is Lipschitz continuous in $p$. Therefore, there exists a unique trajectory that passes through the point $(p_2,s_2)$ for $p\in[p_2-\delta_2, p_2+\delta_2]$, where $\delta_2 >0$ is some constant. 


\noindent
\textbf{Step 2:} In this step, we show that, there exists points $(p,s)$ where the two trajectories described by \eqref{App:Prop3_ODE} and \eqref{App:Prop3_ODE_2} intersect. Furthermore, the derivatives of these two trajectories at the intersection are identical. 

If we multiply \eqref{App:Prop3_ODE_2} with $(N_S-1)(p_f -\mathbb{E}[p_s^*|Q_i<S_i(p_f)]\mathbb{P}(Q_i<S_i(p_f))$ and \eqref{App:Prop3_ODE} with $(p_f -\mathbb{E}[p_s^*])$ and add them together, we obtain
\begin{align}
&(N_S-1)(p_f - \mathbb{E}[p_s^*|Q_i<S_i(p_f)]\mathbb{P}(Q_i<S_i(p_f)))\bigg(C(p_f) + (p_f -\mathbb{E}[p_s^*])C'(p_f) + \mathbb{E}[(\bar{C}(p_s^*)-C(p_f))(p_s^*)']\bigg) \notag\\
&+\underbrace{(p_f -\mathbb{E}[p_s^*])\bigg[(-p_f + \mathbb{E}[p_s^*|Q_i<S_i(p_f)]\mathbb{P}(Q_i<S_i(p_f)))N_SC'(p_f) + \mathbb{E}[(p_s^*)'(p_s^*)\sum_{i=1}(\bar{S}_i(p_s^*))']\bigg]}_{(a)} \notag\\ 
&+\bigg((N_S-1)(p_f -\mathbb{E}[p_s^*|Q_i<S_i(p_f)]\mathbb{P}(Q_i<S_i(p_f)))+ p_f - \mathbb{E}[p_s^*]\bigg)\bigg(\sum_{i=1}^{N_S}S_i(p_f) + \mathbb{E}[(p_s^*)'\sum_{i=1}^{N_S}\bar{S}_i(p_s^*)\notag\\ 
&+(Q_i-S_i(p_f))^-(p_s^*)']\bigg) = 0.  \label{App:Prop3_Cond}
\end{align}
By construction, any solution to \eqref{App:Prop3_Cond} satisfies both \eqref{App:Prop3_ODE} and \eqref{App:Prop3_ODE_2}. 
Then pick any $u>0$. Consider two values that are below and above $u$, respectively, i.e., $\underline{u}<u<\bar{u}$. Denote the respective prices corresponding to these two supplies as $\underline{p}_f$ and $\bar{p}_f$. For $\underline{u}$, let $\underline{p}_f \rightarrow 0^+$ such that $C'(\underline{p}_f) =C(\underline{p}_f)=0$ and
\begin{equation*}
\underline{p}_f -\mathbb{E}[p_s^*|Q_i<\underline{u}]\mathbb{P}(Q_i<\underline{s}) + \underline{p}_f -\mathbb{E}[p_s^*] <0.
\end{equation*}
Then if we let $\underline{u}\rightarrow u^-$, we have $(p_s^*)'\rightarrow 0$ according to Lemma \ref{App:p_s_derivative}. With these conditions, \eqref{App:Prop3_Cond} reduces to
\begin{equation*}
N_S((N_S-1)(\underline{p}_f - \mathbb{E}[p_s^*|Q_i<\underline{u}]\mathbb{P}(Q_i<\underline{u}))+\underline{p}_f -\mathbb{E}[p_s^*])\underline{u}<0.
\end{equation*}
Therefore, we have a negative LHS of \eqref{App:Prop3_Cond} in such case. We next show that the LHS can also be positive. For $\bar{u}$, we can choose $\bar{p}_f$ sufficiently close to $\mathbb{E}[p_s^*]$ such that 
\begin{equation*}
\bar{p}_f -\mathbb{E}[p_s^*|Q_i<\bar{u}]\mathbb{P}(Q_i<\bar{u}) + \bar{p}_f -\mathbb{E}[p_s^*] > 0.  
\end{equation*}
Let $\bar{u}\rightarrow u^+$, we have $(p_s^*)'\rightarrow 0$ according to Lemma \ref{App:p_s_derivative}. Using these results, we can check that the LHS of \eqref{App:Prop3_Cond} becomes positive because term $(a)$ in \eqref{App:Prop3_Cond} is zero and all other terms are positive.

Combining results for $\underline{u}$ and $\bar{u}$, there must exist a price $p\in(\underline{p}_f, \bar{p}_f)$ that makes \eqref{App:Prop3_Cond} hold. In other words, there exists a point that satisfies both \eqref{App:Prop3_ODE} and \eqref{App:Prop3_ODE_2}. Moreover, $S'$ implied by these two ODEs at this point is identical by construction. Together with the results proved in Step 1, this proves the existence of a common solution to both \eqref{App:Prop3_ODE} and \eqref{App:Prop3_ODE_2}. 

In addition, $S'>0$ as $\underline{p}_f\rightarrow 0^+$ and $\underline{u}\rightarrow u^-$ and $\bar{p}_f\rightarrow \mathbb{E}[p_s^*]^-$ and $\bar{u}\rightarrow u^+$.
Also note that $S(0)=0$ and $S'(0)=0$ satisfy both \eqref{App:Prop3_ODE}
and \eqref{App:Prop3_ODE_2}.

\item The proof in this part follows a similar spirit to Proposition 3 in \cite{sunar2019strategic}. To show the uniqueness of the SFE, we solve the ODE \eqref{App:Prop3_ODE} for specific values of $p_f$. Note that $\mathbb{E}[p_s^*|Q_i<S_i(p_f)]\mathbb{P}(Q_i<S_i(p_f))$ and $\mathbb{E}[(Q_i-S_i(p_f))^-(p_s^*)']$ are high order infinitesimal of $p_f$, i.e., we have 
\[\lim_{p_f\rightarrow 0^+}\frac{\mathbb{E}[p_s^*|Q_i<S_i(p_f)]\mathbb{P}(Q_i<S_i(p_f))}{p_f}=0
\]
and
\[
\lim_{p_f\rightarrow 0^+}\frac{\mathbb{E}[(Q_i-S_i(p_f))^-(p_s^*)']}{p_f}= 0.
\]
Under these conditions, \eqref{App:Prop3_ODE} reduces to
\begin{align}
S_i'(p_f) + \frac{1}{(N_S-1)p_f}S_i(p_f) = \underbrace{\frac{1}{N_S(N_S-1)p_f}(-\mathbb{E}[(p_s^*)'\sum_{i\in\mathcal{S}}\bar{S}_i(p_s^*)] - \mathbb{E}[(p_s^*)'p_s^*\sum_{i\in\mathcal{S}}\bar{S}'_i(p_s^*)]) -\frac{1}{N_S-1}C'(p_f)}_{G(p_f;p_s^*,\bar{S}_i^*)}\notag.
\end{align}
The solution to the above linear ODE is
\begin{equation}
\label{App:Sol_DA}
S_i(p_f)= (p_f)^{\frac{1}{N_S-1}}\bigg(c_s -\int_0^{p_f} (x)^{-\frac{1}{N_S-1}}G(x)dx\bigg),
\end{equation}
where $c_s\in\mathcal{R}$ is a constant. Based on the solution, the market clearing price $p_f^*$ depends on $c_s$. Let $p_f^*(c_s)$ denote the DA market clearing price as a function of $c_s$. 
For a given demand $D_l$, the market clearing condition \eqref{RT:MarketClearing} implies that a steeper $S_i$ results in a smaller market clearing price, and therefore, $p^\ast_s(c_s)$ is decreasing in $c_s$.

Next, consider the payoff of the $i$-th renewable supplier as a function of the constant $c_s$
\begin{align*} 
\Pi^{\mbox{DA}}_i(c_s) &=\Big(D_{l}-C(p_{f}^*(c_s))-S_{-i}(p_{f}^*(c_s);c_s)\Big)p_{f}^*(c_s)\\
 & + \mathbb{E}\Big[\bar{S}_{i}^*(p_s^*(p_f^*(c_s)))p_s^*(p_f^*(c_s))\Big] + \mathbb{E}\Big[\big(Q-S_i(p_f^*(c_s))\big)^-p_s^*(p_f^*(c_s))\Big].
 \end{align*}
Taking derivative w.r.t. $c_s$, we obtain
{\allowdisplaybreaks
\begin{align*}
\frac{\mathrm{d}\Pi^{\mbox{DA}}_i(c_s)}{\mathrm{d}c_s} =
 & \ \Big(D_{l}-C(p_{f}^*(c_s))-S_{-i}(p_{f}^*(c_s);c_s)\Big)\frac{\mathrm{d}p_{f}^*(c_s)}{\mathrm{d}c_s} - p_f^*(c_s)\Big(C'(p_f^*(c_s)) + S_{-i}'(p_f^*(c_s);c_s)\Big)\frac{\mathrm{d}p_{f}^*(c_s)}{\mathrm{d}c_s}\\
&\mbox{}- p_f^*(\bar{c}_1) \frac{\partial S^*_{-i}(p_f^*(c_s);c_s)}{\partial c_s} - \mathbb{E} \bigg[
     \bigg(\bar{S}_i^*(p_s^*(p_f^*(c_s)))\bigg)'\Big(p_s^*(p_f^*(c_s))\Big)'\frac{\mathrm{d}p_{f}^*(c_s)}{\mathrm{d}c_s}p_s^*(p_f^*(c_s))\bigg]\\
   & \mbox{} + \mathbb{E} \bigg[\bar{S}_i^*(p_s^*(p_f^*(c_s)))\Big(p_s^*(p_f^*(c_s))\Big)'\frac{\mathrm{d}p_{f}^*(c_s)}{\mathrm{d}c_s}\bigg]\\
     &\mbox{} +\mathbb{E}[Q_i\big(p_s^*(p_f^*(c_s))\big)'|Q_i<D_l^r(p_f^*)]\mathbb{P}(Q_i<D_l^r(p_f^*))\frac{\mathrm{d}p_{f}^*(c_s)}{\mathrm{d}c_s}
\end{align*}
Rearranging, we have
\begin{align*}
\frac{\mathrm{d}\Pi^{\mbox{DA}}_i(c_s)}{\mathrm{d}c_s} =& \ \frac{\mathrm{d}p_{f}^*(c_s)}{\mathrm{d}c_s}\bigg\{\Big(D_{l}-C(p_{f}^*(c_s))-S_{-i}(p_{f}^*(c_s);c_s)\Big)\\
   & \mbox{} - \Big(C'(p_f^*(c_s)) + S_{-i}'(p_f^*(c_s);c_s)\Big)p_f^*(c_s)\\
   & \mbox{} + \mathbb{E} \bigg[\bigg(\bar{S}_i^*(p_s^*(p_f^*(c_s)))\bigg)'\Big(p_s^*(p_f^*(c_s))\Big)'p_s^*(p_f^*(c_s))\bigg]\\
   & \mbox{} + \mathbb{E} \bigg[\bar{S}_{i}^*\Big(p_s^*(p_f^*(c_s))\Big)
     \Big(p_s^*(p_f^*(c_s))\Big)'\bigg] \\
     &+\mathbb{E}[Q_i\big(p_s^*(p_f^*(c_s))\big)'|Q_i<D_l^r(p_f^*)]\mathbb{P}(Q_i<D_l^r(p_f^*))\bigg\}\\
     & -  p_f^*(c_s)\sum_{j\neq i}\frac{\partial S^*_{j}(p_f^*(c_s);c_s)}{\partial c_s},
\end{align*}
}
 where we used the notation $\bar{S}_{-i}(p_s^*) = \sum_{j\in\mathcal{S}\setminus\{i\}}S_j(p_s^*)$ given in Eq.~\eqref{eq:supplmini}. 
The terms within the curly brackets $\{\cdot\}$ represent the FOC of the supplier's payoff (w.r.t. $p_f^*$) and therefore is zero.
Besides, from \eqref{App:Sol_DA}, we know that 
\[\frac{\partial S^*_{i}(p_f^*(c_s);c_s)}{\partial c_s}= (p_f^*(c_s))^{1/(N_S-1)}>0.
\]
Consequently, the above derivative $\frac{\mathrm{d}\Pi^{\text{ DA}}_i(c_s)}{\mathrm{d}c_s}$ is always negative, indicating that the payoff is decreasing in $c_s$. Thus, the renewable supplier would choose the smallest $c_s$ that makes the supply function increasing. By following a similar argument in Proposition \ref{RT:Uniqueness}, we conclude that the constant $c_s$ is uniquely determined.
\end{enumerate}
Combining (i) and (ii), we obtain the stated results.
\end{proof}

\DAPriceGapVanilla*
\begin{proof}
Let $q(t)$ denote the LSE's DA demand bid  for $t\in\mathcal{T}$. Since the LSE
is the sole strategic player in the considered case, its payoff is 
\begin{equation}
C^{-1}\big(q(t)\big)q(t) + \mathbb{E}\Big[\big(D(t)-q(t)\big)^{+}\bar{C}^{-1}(D(t))\Big],
\end{equation}
where the first term is the immediate cost in the DA market and second
one is the expected cost in the RT market. The gradient w.r.t. $q(t)$ is
given by
\[
  C^{-1}(q(t)) + \big(C^{-1}(q(t))\big)'q(t)
  -\mathbb{E}\big[\bar{C}^{-1}(D(t))\big|D(t)>q(t)\big]\mathbb{P}(D(t)>q(t)).
\]
At $q(t)=0$, the 
gradient 
is 
strictly 
negative because the conditional expectation and the probability are
positive, and 
$C^{-1}(0)=0$. 
Thus, it follows that 
the optimal demand bid $q^*(t)$ is
strictly positive.

Moreover, since the market participants only include the LSE and the
conventional supplier, the DA market clearing price  $p_f^*(t) =
C^{-1}(q^*(t))$.  
Similarly,
\begin{equation*}
  \mathbb{E}[p_s^*(t)] = 
\mathbb{E}\big[\bar{C}^{-1}(D(t))\big|D(t)>q(t)\big]\mathbb{P}(D(t)>q(t)), 
\end{equation*}
since 
$p_s^*(t)=0$ when $D(t)\leq q(t)$.
These together imply that the FOC is equivalent to
\[
  p_f^*(t) + \big(C^{-1}(q^*(t))\big)'q^*(t) -\mathbb{E}\big[p_s^*(t)\big]
  =0.
\]
Since $C(\cdot)$ is strictly increasing, it follows that $\big(C^{-1}(q^*(t))\big)'q^*(t) >0$, 
and therefore, 
$p_f^*(t)<\mathbb{E}[p_s^*(t)]$.
\end{proof}
We need the following result to 
prove Proposition~\ref{DA:PriceAlignment_Vanilla}, 

\begin{lemma}
\label{DA:BoundedVB}
In a two settlement electricity market that only includes the LSE and the conventional supplier, virtual DEC traders are incentivized to join the market. Furthermore, the total virtual load in the DA market is always finite even as the number of virtual bidders goes to infinity, i.e.,
\begin{equation*}
\lim_{N_V\rightarrow+\infty}\sum_{v=1}^{N_V}B_v(p_f^*) <\infty.
\end{equation*} 
\end{lemma}
\begin{proof}
Proposition~\ref{DA:PriceGap_Vanilla} shows that
$p_f^*<\mathbb{E}[p_s^*]$ when only the conventional supplier and the LSE
are in the market. In this scenario, virtual DEC bidders can generate a positive
payoff by 
purchasing electricity from the DA market, and
selling the same amount to the RT market. Since the DEC bids
increase the total demand in the DA market, $p_f^*$ increases; and since
DEC bids reduce residual demand in the RT market, the $\mathbb{E}[p_s^*]$
increases.
Throughout this process, the condition
$p_f^*\leq\mathbb{E}[p_s^*]$ always holds, becasue the DEC bidders exit the
market when $p_f^*>\mathbb{E}[p_s^*]$.\footnote{There are no INC bidders
  during such process, because their payoffs would be negative.} 

Suppose there are $N_V$ DEC bidders in the market and the condition
$p_f^*\leq\mathbb{E}[p_s^*]$ holds. The total virtual load cannot exceed
the DA total 
supply because the conventional supplier is the sole supplier:
\begin{equation*}
\sum_{v=1}^{N_V}B_v(p_f^*)\leq C(p_f^*).
\end{equation*}
It follows that
\begin{equation*}
\sum_{v=1}^{N_V}B_v(p_f^*)\leq C(p_f^*)\leq C(\mathbb{E}[p_s^*])<+\infty,
\end{equation*}
where the second inequality holds due to $p_f^*\leq\mathbb{E}[p_s^*]$ and the monotonicity of the supply functions. 
The last inequality is true because $\mathbb{E}[p_s^*]<+\infty$, given
that the residual demand $D_r\in [\underline{D}_r,\overline{D}_r]$ is
bounded and the conventional supply function is strictly increasing when
$p_s^*>p_f^*$. 
We obtain the stated result by letting $N_V\rightarrow +\infty$.\looseness=-1 
\end{proof}

\DAPriceAlignmentVanilla*
\begin{proof}
Using the market clearing condition \eqref{DA:MarketClearing} and \eqref{DA:VB_Obj}, the FOC of the $v$-th virtual trader's payoff is
\begin{align*}
\big(\mathbb{E}[\big(p_s^*\big)^{\prime}]-1\big)\big(C(p_f^*)-D_l-B_{-v}(p_f^*)\big)=\big(\mathbb{E}[p_s^*]-p_f^*\big)\big((B_{-v}'(p_f^*)-C'(p_f^*)\big). 
\end{align*}
Simplifying using the DA market clearing condition \eqref{DA:MarketClearing} yields:
\begin{equation*}
\big(\mathbb{E}[\big(p_s^{*}\big)']-1\big)B_v(p_f^*) = \big(\mathbb{E}[p_s^*]-p_f^*\big)\big(B_{-v}'(p_f^*)-C'(p_f^*)\big).
\end{equation*}
Summing over $v$ and taking the limit as $N_V\rightarrow\infty$ on both sides, we have
\begin{align*}
\lim_{N_V\rightarrow+\infty}\big(\mathbb{E}[\big(p_s^*\big)^{\prime}]-1\big)
\frac{1}{N_V}\sum_{v=1}^{N_V}B_v(p_f^*)
= \lim_{N_V\rightarrow+\infty}\big(\mathbb{E}[p_s^*]-p_f^*\big)\underbrace{\Big(\frac{N_V-1}{N_V}\sum_{v=1}^{N_V}B_{v}'(p_f^*)-C'(p_f^*)\Big)}_{(a)}.
\end{align*}
By Lemma \ref{App:p_s_derivative} and \ref{DA:BoundedVB},
$\mathbb{E}[(p_s^*)']$ and
$\lim_{N_V\rightarrow\infty}\sum_{v=1}^{N_V}B_v(p_f^*)$ are bounded, so
the left hand side approaches zero as $1/N_V\rightarrow 0$. On the right hand side, the parenthetical term $(a)$ can not be zero because $\frac{N_V-1}{N_V}\sum_{v=1}^{N_V}B_{v}'(p_f^*)\leq 0$ and $-C'(p_f^*)<0$. Thus, 
we have
$
p_f^* = \mathbb{E}[p_s^*].
$
in the limit
as $N_V\rightarrow\infty$.
\end{proof}

\DAPriceGap*
\begin{proof}
\label{App:Prop4_PriceGap} 
Using the market clearing condition \eqref{DA:MarketClearing}, the LSE's payoff \eqref{DA:LSE_Obj} can be rewritten as:
\begin{equation*}
\bigg[\sum_{i\in\mathcal{S}}S_i(p_f^*)+C(p_f^*)\bigg]p_f^*+\mathbb{E}\bigg[\bigg(D- (S_i(p_f^*)+C(p_f^*))\bigg)^+p_s^*\bigg].
\end{equation*}
The FOC with respect to $p_f^*$ is
\begin{align*}
&\bigg[\sum_{i\in\mathcal{S}}S_i(p_f^*)+C(p_f^*)\bigg]+\bigg[\sum_{i\in\mathcal{S}}S'_i(p_f^*)+C'(p_f^*)\bigg]p_f^*\notag\\
&+\mathbb{E}\bigg[\bigg(-\sum_{i\in\mathcal{S}}S'_i(p_f^*)-C'(p_f^*)\bigg)p_s^*\bigg|D>D_l\bigg]\mathbb{P}\bigg(D>D_l\bigg)\notag\\
&+\mathbb{E}\bigg[\bigg(\sum_{i\in\mathcal{S}}\bar{S}_i(p_s^*) +\sum_{i\in\mathcal{S}}(Q_i-S_i(p_f^*))^-+\bar{C}(p_s^*)-C(p_f^*)\bigg)^+p_s^*\bigg] = 0,
\end{align*}
where the last term follows from \eqref{RT:MarketClearing}.
Rearranging the above FOC, we obtain
\begin{align}
&\sum_{i\in\mathcal{S}}S_i(p_f^*)+(p_f^*-\mathbb{E}[p_s^*])\sum_{i\in\mathcal{S}}S'_i(p_f^*) + \mathbb{E}\bigg[(p_s^*)'\sum_{i\in\mathcal{S}}\big(\bar{S}_i(p_s^*)+(Q_i-S_i(p_f^*))^-\big)\bigg|D>D_l\bigg]\mathbb{P}\bigg(D>D_l\bigg)\notag\\ 
&+C(p_f^*)+(p_f^*-\mathbb{E}[p_s^*])C'(p_f^*)+\mathbb{E}\bigg[(p_s^*)'(\bar{C}(p_s^*)-C(p_f^*))\bigg|D>D_l\bigg]\mathbb{P}\bigg(D>D_l\bigg) = 0. \label{App:Prop6_LSE_FOC}
\end{align}
Meanwhile, the FOC of renewable suppliers' payoff \eqref{DA:RS_Obj} is
\begin{align*}
&S_i(p_f^*) + (-C'(p_f^*)-S_{-i}'(p_f^*))p_f^* + \mathbb{E}[(p_s^*)'p_s^*\bar{S}_i'(p_s^*)] + \mathbb{E}[(p_s^*)'\bar{S}_i(p_s^*)]\\
&+\mathbb{E}\bigg[(Q_i -D_l+C(p_f^*)+S_{-i}(p_f^*))(p_s^*)'\bigg|Q_i<S_i(p_f^*)\bigg]\mathbb{P}\bigg(Q_i<S_i(p_f^*)\bigg) \\
&+ \mathbb{E}\bigg[(S_{-i}'(p_f)+C'(p_f))p_s^*\bigg|Q_i<S_i(p_f^*)\bigg]\mathbb{P}\bigg(Q_i<S_i(p_f^*)\bigg)= 0 
\end{align*}
Summing over $i$ and rearranging give:
\begin{align*}
&\sum_{i\in\mathcal{S}}S_i(p_f^*) + \bigg(p_f^* - \mathbb{E}\big[p_s*|Q_i<S_i(p_f^*)\big]\mathbb{P}(Q_i<S_i(p_f^*))\bigg)\bigg((N_S-1)\sum_{i\in\mathcal{S}}S'_i(p_f^*) +N_SC'(p_f^*)\bigg)\\ 
& +\mathbb{E}[(p_s^*)'\sum_{i\in\mathcal{S}}\bar{S}_i(p_s^*)]+ \mathbb{E}[(p_s^*)'p_s^*\sum_{i\in\mathcal{S}}\bar{S}'_i(p_s^*)]+ \mathbb{E}\bigg[(p_s^*)'\sum_{i\in\mathcal{S}}(Q_i-S_i(p_f^*))^-\bigg|D>D_l\bigg]\mathbb{P}\bigg(D>D_l\bigg)\notag =0.
\end{align*}
Using the above suppliers' FOC to simplify \eqref{App:Prop6_LSE_FOC}, we obtain
\begin{align}
&\bigg(p_f^* - \mathbb{E}[p_s^*]\bigg)\sum_{i\in\mathcal{S}}S_i'(p_f^*) + \bigg(p_f^*-\mathbb{E}[p_s^*|Q_i<S_i(p_f^*)]\mathbb{P}(Q_i<S_i(p_f^*))\bigg)\bigg((N_S-1)\sum_{i\in\mathcal{S}}S_i'(p_f^*)+N_SC'(p_f^*)\bigg)\notag\\
&+C(p_f^*) + \bigg(p_f^* - \mathbb{E}[p_s^*]\bigg)C'(p_f^*)+\mathbb{E}\bigg[(p_s^*)'(\bar{C}(p_s^*)-C(p_f^*))\bigg|D>D_l\bigg]\mathbb{P}\bigg(D>D_l\bigg) - \mathbb{E}[(p_s^*)'\sum_{i\in\mathcal{S}}\bar{S}'_i(p_s^*)]=0.\label{App:Lemma5_Cond}
\end{align}
Recall that ${p^*_s}'$ is negative from Lemma~\ref{App:p_s_derivative} and $S_i(\cdot)$ is an increasing function for any $i$ by Assumption (see Section~\ref{sec:DAMarket}). If $p_f^*\geq \mathbb{E}[p_s^*]$, then all terms in \eqref{App:Lemma5_Cond} are non negative except:
\begin{equation}
\label{DA:CS_Cond}
C(p_f^*) + \bigg(p_f^* - \mathbb{E}[p_s^*]\bigg)C'(p_f^*)+\mathbb{E}\bigg[(p_s^*)'(\bar{C}(p_s^*)-C(p_f^*))\bigg|D>D_l\bigg]\mathbb{P}\bigg(D>D_l\bigg).
\end{equation}
Next, we will show that \eqref{DA:CS_Cond} is also strictly positive. Thus, if $p_f^* \geq \mathbb{E}[p_s^*]$, \eqref{App:Lemma5_Cond} can not be zero. This indicates that all equilibriums must make $p_f^* < \mathbb{E}[p_s^*]$.

To demonstrate that \eqref{DA:CS_Cond} is positive, it sufficies to show that increasing $p_f^*$ further increases the conventional supplier's payoff
\begin{equation*}
C(p_f^*)p_f^* +\mathbb{E}[(\bar{C}(p_s^*)-C(p_f^*))p_s^*].
\end{equation*}
This is true because increasing the equilibrium DA price $p_f^*$ will increase the procurement cost of the LSE and decrease the payoff of the renewable suppliers as they deviate from the equilibrium. Consequently, the total payoff of the conventional supplier increases, as the LSE's payment is allocated to these two suppliers. This implies that the first derivative of the conventional supplier's payoff at $p_f^*$ is always positive:
\begin{equation*}
C(p_f^*) + \bigg(p_f^* - \mathbb{E}[p_s^*]\bigg)C'(p_f^*)+\mathbb{E}\bigg[(p_s^*)'(\bar{C}(p_s^*)-C(p_f^*))\bigg|D>D_l\bigg]\mathbb{P}\bigg(D>D_l\bigg) + \mathbb{E}[\bar{C}'(p_s^*)(p_s^*)'p_s^*] \geq 0.
\end{equation*}
The last term in the above formula is negative due to $(p_s^*)'<0$. Thus, \eqref{DA:CS_Cond} is strictly positive.
\end{proof}
\DAPriceAlignmentOne*
\begin{proof}
The proof is similar to that of Proposition \ref{DA:PriceAlignment_Vanilla}. According to \eqref{DA:MarketClearing} and \eqref{DA:VB_Obj}, the FOC of the virtual trader's payoff w.r.t $p_f^*$ is
\begin{align*}
&\bigg(\mathbb{E}[\big(p_s^*\big)^{\prime}]-1\bigg)\bigg(\sum_{i\in\mathcal{S}}S_i(p_f^*)+C(p_f^*)-D_l-B_{-v}(p_f^*)\bigg)\notag\\
&\hspace{4cm}=\bigg(\mathbb{E}[p_s^*]-p_f^*\bigg)\bigg(B_{-v}'(p_f^*)-\sum_{i\in\mathcal{S}}S_i'(p_f^*)-C'(p_f^*)\bigg).
\end{align*}
Using the DA market clearing condition \eqref{DA:MarketClearing}, we can simplify it as
\begin{equation*}
\bigg(\mathbb{E}[\big(p_s^{*}\big)']-1\bigg)B_v(p_f^*) = \bigg(\mathbb{E}[p_s^*]-p_f^*\bigg)\bigg(B_{-v}'(p^*_f)-\sum_{i\in\mathcal{S}}S_i'(p^*_f)-C'(p^*_f)\bigg).
\end{equation*}
Summing over $v$ and taking the limit as $N_S\rightarrow\infty$ on both sides yield
\begin{align*}
&\lim_{N_V\rightarrow+\infty}\bigg(\mathbb{E}[\big(p_s^*\big)^{\prime}]-1\bigg)\frac{1}{N_V}\sum_{v=1}^{N_V}B_v(p_f^*)\\
=& \lim_{N_V\rightarrow+\infty}\bigg(\mathbb{E}[p_s^*]-p_f^*\bigg)\underbrace{\bigg(\frac{N_V-1}{N_V}\sum_{v=1}^{N_V}B_{v}'(p_f^*)-\sum_{i\in\mathcal{S}}S_i'(p_f^*)-C'(p_f^*)\bigg)}_{(a)}.
\end{align*}
As $N_V\rightarrow +\infty$, the left hand side is zero because $\mathbb{E}[(p_s^*)']$ and $\sum_{v=1}^{N_V}B_{v}(p_f^*)$ are bounded due to Lemma \ref{App:p_s_derivative} and \ref{DA:BoundedVB}. On the right hand side, the parenthetical term $(a)$ can not be zero because $-\frac{N_V-1}{N_V}\sum_{v=1}^{N_V}B_{v}'(p_f^*)\leq 0$ and $-\sum_{i\in\mathcal{S}}S_i'(p_f^*) -C'(p_f^*) < 0$ for $p_f^*>0$.\footnote{If $p_f^*>\underline{p}_c$, $-C'(p_f^*) < 0$; If $0\leq p_f^* \leq\underline{p}_c$, $-\sum_{i\in\mathcal{S}}S_i'(p_f^*)<0$ because no equilibrium supply function will reach the expected capacity limit on this price interval.} Thus, we conclude with 
\begin{equation*}
p_f^* = \mathbb{E}[p_s^*]
\end{equation*}
as $N_V\rightarrow+\infty$.
\end{proof}

\DAExistenceTwo*
\begin{proof}
Consider the payoff of the LSE
\begin{equation*}
\bigg[\sum_{i\in\mathcal{S}}S_i(p_f^*)+C(p_f^*)-\sum_{v=1}^{\infty}B_v(p_f^*)\bigg]p_f^*+\mathbb{E}\bigg[\bigg(D- (\sum_{i\in\mathcal{S}}S_i(p_f^*)+C(p_f^*)-\sum_{v=1}^{\infty}B_v(p_f^*))\bigg)^+p_s^*\bigg].
\end{equation*}
The FOC with respect to $p_f^*$ is
\begin{align*}
&\bigg[\sum_{i\in\mathcal{S}}S_i(p_f^*)+C(p_f^*)-\sum_{v=1}^{\infty}B_v(p_f^*)\bigg]+\bigg[\sum_{i\in\mathcal{S}}S'_i(p_f^*)+C'(p_f^*)-\sum_{v=1}^{\infty}B'_v(p_f^*)\bigg]p_f^*\notag\\
&+\mathbb{E}\bigg[\bigg(-\sum_{i\in\mathcal{S}}S'_i(p_f^*)-C'(p_f^*)+\sum_{v=1}^{\infty}B'_v(p_f^*)\bigg)p_s^*\bigg|D>D_l\bigg]\mathbb{P}\bigg(D>D_l\bigg)\\
&+\mathbb{E}\bigg[\bigg(D- (S_i(p_f^*)+C(p_f^*)-\sum_{v=1}^{\infty}B_v(p_f^*))\bigg)^+(p_s^*)'\bigg] = 0,
\end{align*}
Using the price alignment condition $p_f = \mathbb{E}[p_s^*]$ and rearranging, we obtain
\begin{equation*}
(1-\mathbb{E}[(p_s^*)'])\bigg(\sum_{i=1}^{N_S}S_i(p_f)+C(p_f)-\sum_{v=1}^{N_V}B_v(p_f)\bigg)+\mathbb{E}[D(p_s^*)']=0.
\end{equation*}
which is equivalent to 
\begin{equation*}
D_l = \frac{\mathbb{E}[(p_s^*)'D]}{\mathbb{E}[(p_s^*)'] - 1}.
\end{equation*}
according to the market clearing condition. Moreover, using the FOC of renewable supplier's payoff function and the price alignment condition, we obtain the following equation
\begin{align*}
&S_i(p_f) +\int_{0}^{S_i(p_f)}(x-S_i(p_f))(p_s^*)'h(x)dx \\
&+(p_f-\mathbb{E}[p_s^*|Q_i<S_i(p_f)]\mathbb{P}(Q_i<S_i(p_f)))(\sum_{v=1}^{\infty}B_v(p_f)-S_{-i}'(p_f)-C(p_f))\\
&+\mathbb{E}[(\bar{S}_i^*(p_s^*))'(p_s^*)'(p_s^*)+\bar{S}_i^*(p_s^*)(p_s^*)']=0.
\end{align*}
Rearranging, we obtain
\begin{align*}
S'_i(p_f) &= \frac{1}{N_S-1}\bigg[\frac{1}{p_f - \mathbb{E}[p_s^*|Q_i<S_i(p_f)]\mathbb{P}(Q_i<S_i(p_f))}\bigg(-S_i(p_f)-\int_{0}^{S_i(p_f)}(x-S_i(p_f)(p_s^*)'h(x)dx\\
&-\mathbb{E}[(\bar{S}_i^*(p_s^*))'(p_s^*)'(p_s^*)+\bar{S}_i^*(p_s^*)(p_s^*)'])\bigg)-\sum_{v=1}^{\infty}B'_v(p_f)-C'(p_f)\bigg].
\end{align*}
Note that since $ - \sum_{v=1}^{\infty}B'_v(p_f)\geq 0$ and a ``steep" supply function will decrease the renewable supplier's payoff, as shown in Proposition \ref{RT:Uniqueness}, we have $\sum_{v=1}^{\infty}B'_v(p_f) = 0$ at equilibrium.
Plugging it into the above ODE, we obtain
\begin{align*}
S'_i(p_f) &= \frac{1}{N_S-1}\bigg[\frac{1}{p_f - \mathbb{E}[p_s^*|Q_i<S_i(p_f)]\mathbb{P}(Q_i<S_i(p_f))}\bigg(-S_i(p_f)-\int_{0}^{S_i(p_f)}(x-S_i(p_f)(p_s^*)'h(x)dx\\
&-\mathbb{E}[(\bar{S}_i^*(p_s^*))'(p_s^*)'(p_s^*)+\bar{S}_i^*(p_s^*)(p_s^*)'])\bigg)-C'(p_f)\bigg].
\end{align*}
This is identical to \eqref{App:Prop3_ODE}. Therefore, the stated results hold.
\end{proof}

\DAUniquenessDemand*
\begin{proof}
Suppose $D_l(t)$ and $D_l^*(t)$ are two optimal DA demands that yield the same cost for the LSE. Without loss of generality, we assume $D_l(t)>D_l^*(t)$. Suppose also that the LSE decreases its DA demand from $D_l(t)$ to $D_l^*(t)$. In response, let the DEC bidders increase their trading volumes to compensate for the demand shortfall and equalize the DA and RT market prices. Then the renewable suppliers would maintain their supply functions in both markets, as the total demands in these two markets are identical to the situation before the LSE adjusts its demand. Note that after these adjustments, a new equilibrium is established, because the payoffs for each renewable supplier and virtual trader are unchanged, and $D_l^*(t)$ is assumed to be optimal for the LSE.  So they would not deviate from their decisions. 

However, in the new equilibrium, the LSE's DA payment decreases, as it purchases less DA demand at an identical DA price. Its expected RT payment remains the same, because the supply functions and the residual demand in the RT market are unaffected as stated earlier. Consequently, the LSE's total payment decreases. This contradicts to the statement that $D_l(t)$ and $D_l^*(t)$ generate the same cost for the LSE. Thus, there would be no multiple optimal DA demand solutions.
\end{proof}

\end{document}